\documentclass[12pt,a4paper,oneside]{article}
\usepackage{geometry}
\usepackage{amsmath,amssymb,amsthm}
\usepackage[section]{placeins}
\usepackage{graphicx,color}
\usepackage{hyperref}
\usepackage{float}
\usepackage{subfig}
\usepackage{dsfont}
\usepackage{capt-of}
\usepackage{mathtools}
\usepackage{stmaryrd}
\usepackage{authblk}
\usepackage{booktabs}
\usepackage{listings}

\usepackage{epstopdf}
\usepackage{ mathrsfs }

\usepackage{lipsum} 
\usepackage{algorithm}
\usepackage{algpseudocode}

\newtheorem{theorem}{Theorem}
\newtheorem{lemma}{Lemma}
\newtheorem{corollary}{Corollary}

\newtheorem{definition}{Definition}

\newtheorem{proposition}{Proposition}

\newcommand{\R}{\mathbb{R}}
\begin{document}
\title{Support identification for parameter variations in a PDE system via regularized methods}
\author[1]{Houcine  Meftahi \thanks{houcine.meftahi@enit.utm.tn}}
\author[2]{Chayma Nssibi \thanks{chayma.nssibi@enit.utm.tn}}
\affil[1]{University of Jendouba/ISIKEF and  ENIT of Tunis, Tunisia.}
\affil[2]{University of Tunis Elmanar, ENIT-LAMSIN, B.P. 37, 1002 Tunis}
\maketitle

\begin{abstract}
We study the inverse problem of recovering the spatial support of parameter variations in a system of partial differential equations (PDEs) from boundary measurements. A reconstruction method is developed based on the monotonicity properties of the Neumann-to-Dirichlet operator, which provides a theoretical foundation for stable support identification. To improve reconstruction accuracy, particularly when parameters have disjoint  supports, we propose a combined regularization approach integrating monotonicity principles with Truncated Singular Value Decomposition (TSVD) regularization. This hybrid strategy enhances robustness against noise and ensures sharper support localization. Numerical experiments demonstrate the effectiveness of the proposed method, confirming its applicability in practical scenarios with varying parameter configurations.
\end{abstract}

\section{Introduction}
In this article, we investigate an inverse problem involving the reconstruction of the support of parameter variations for the following system of partial differential equations:
\begin{equation}\label{system}
- \nabla \cdot \left( \lambda (\nabla \cdot {u}) {I} + 2\mu \nabla^s {u} \right) + \rho {u} = \mathbf{0} \quad \text{in } \Omega,
\end{equation}
where $\lambda$ and $\mu$ are the Lamé parameters, characterizing the elastic properties of the material, $\rho$ denotes the  density, ${u}$ represents the displacement field, $\nabla^s {u}$ is the symmetric gradient (strain tensor), and $\Omega$ is the spatial domain of interest.

This PDE system can be viewed  as an approximate  linear model   in biological tissues, where the term $\rho {u}$ corresponds to a restoring force arising from tissue elasticity. Such a  model  is considered in \cite{lin2006quantitative,meftahi2025stability,yu2011strong}.

The study of this inverse problem is highly relevant in medical diagnostics, particularly for detecting anomalies in biological tissues. Variations in the Lamé parameters $\lambda$ and $\mu$ as well as the density $\rho$ can indicate pathological changes, such as the presence of tumors or cancerous lesions. By reconstructing the support of these parameter variations, our approach provides a non-invasive imaging tool that could enhance early disease detection and treatment planning.
  
The inverse coefficients problem for the given PDE system is investigated in \cite{meftahi2025stability}, where the authors establish both non-constructive and constructive Lipschitz stability results. Their analysis leverages advanced techniques, including monotonicity methods and the theory of localized potentials, to derive rigorous stability estimates.

Our primary contribution is numerical. We introduce a regularized reconstruction framework that effectively combines the monotonicity method with the Truncated Singular Value Decomposition (TSVD) approach. Unlike the results presented in \cite{eberle2022monotonicity}, where parameter variations are restricted to identical supports, in this work we treat the general case, and we show that our method can also recover disjoint supports of the parameter perturbations accurately. Furthermore, we establish that the proposed reconstruction remains stable even under significant noise levels, showcasing its robustness in practical applications.

We first address the case where the variations of physical parameters share the same support, with the goal of recovering the shape of this support from boundary measurements, assuming that the parameter values inside and outside the shape are known a priori. Secondly, we consider the case where the parameter variations have disjoint supports.

The first case  is regarded as an anomaly detection inside the medium. In this context, various reconstruction methods have been explored by researchers, employing techniques such as iterative algorithms based on topological and shape gradient approaches. Notable works in this area include \cite{belhachmi2018topology,belhachmi2023level,  belhachmi2013shape,chaabane2013topological, laurain2016shape,meftahi2009etudes}, among others. Additionally, the monotonicity method has proven effective in anomaly detection across diverse settings for problems involving one and two parameters.  

In this work, we demonstrate that the monotonicity method remains effective for inverse problems involving three physical parameters, and that stable numerical reconstruction is achievable. In the case of a single support, we employ a  reconstruction approach based on the linearized monotonicity test. However, this method is sensitive to high levels of noise in the data. To improve robustness, we propose a hybrid strategy that combines the linearized regularized monotonicity method with the Truncated Singular Value Decomposition (TSVD) approach. In contrast to the former method, this combined approach exhibits both accuracy and stability under significant noise and is capable of handling disjoint supports. 
 
The monotonicity method was first introduced in \cite{tamburrino2002new} to solve the electrical resistance tomography problem. By exploiting the monotonicity properties of the resistance matrix, this approach provides a non-iterative reconstruction framework. Since then, the method has been generalized to static imaging modalities, including Electrical Capacitance Tomography, Inductance Tomography, and Electrical Resistance Tomography \cite{calvano2012fast, garde2022simplified, garde2022reconstruction}. It has also been extended to time-dependent problems, such as Magnetic Induction Tomography (MIT) \cite{Tamburrino}.  

The monotonicity method has also been applied to problems in linear elasticity, where both standard and linearized formulations have been investigated \cite{eberle2021shape,eberle2023resolution}. In addition, monotonicity-based techniques have been employed for shape reconstruction in electrical impedance tomography \cite{harrach2013monotonicity}. Recent research further underscores the versatility and effectiveness of the monotonicity method in addressing a wide range of inverse problems \cite{esposito2024piecewise, griesmaier2022inverse, klibanov2025convexification, lin2022monotonicity, mottola2024imaging}.

The paper is structured as follows. Section 2 introduces the direct problem, establishes the Fréchet differentiability of the Neumann-to-Dirichlet operator, and formulates the corresponding inverse problem. Section 3 analyzes the linearized monotonicity property of the Neumann-to-Dirichlet operator and demonstrates its application to shape reconstruction, supported by numerical results. Section 4 reformulates the inverse problem as the minimization of a residual functional constrained by the monotonicity method, with further numerical illustrations. Section 5 combines the monotonicity method with the truncated singular value decomposition (TSVD) approach and presents additional numerical results.

\section{Problem Formulation}
In this section, we begin by presenting the problems of interest: the direct problem and the inverse problem.

Let $\Omega \subset \mathbb{R}^d$ ($d = 2$ or $3$) be a bounded and connected open set with Lipschitz boundary $\partial \Omega = \Gamma_D \cup \Gamma_N$, where $\Gamma_D$ and $\Gamma_N$ represent the Dirichlet and Neumann boundaries, respectively, satisfying $\Gamma_D \cap \Gamma_N = \emptyset$. We assume that both $\Gamma_D$ and $\Gamma_N$ are relatively open and connected.

For the following definitions, we introduce the set:
\begin{equation*}
    L_+^{\infty}(\Omega) := \{ w \in L^{\infty}(\Omega) : \operatorname{essinf}_{x \in \Omega} w(x) > 0 \}.
\end{equation*}

Let $u : \Omega \to \mathbb{R}^d$ be the displacement vector, and let the Lam\'e parameters and the density  be given by $\mu, \lambda, \rho : \Omega \to L_+^{\infty}(\Omega)$. The symmetric gradient is defined as:
\begin{equation*}
    {\nabla}^s u = \frac{1}{2} (\nabla u + (\nabla u)^T).
\end{equation*}
We denote $\nu$  the outward-pointing normal vector on $\partial \Omega$, $g \in L^2(\Gamma_N)^d$ represents the boundary force, and $I$ is the $d \times d$ identity matrix.

The divergence of a matrix $M \in \mathbb{R}^{d \times d}$ is defined as:
\begin{equation*}
    \nabla \cdot M = \sum_{i,j=1}^{d} \frac{\partial M_{ij}}{\partial x_j} e_i,
\end{equation*}

where $e_i$ is a unit vector and $x_j$ is a component of a vector in $\mathbb{R}^d$.

The direct problem in linear elasticity requires finding $u \in H^1(\Omega)^d$ such that
\begin{equation}\label{direct2}
\left\{
\begin{aligned}
   - \nabla \cdot \left( \lambda (\nabla \cdot u) I + 2\mu {\nabla}^s u \right) +\rho u& = 0 \quad \text{in } \Omega.\\
     (\lambda (\nabla \cdot u) I + 2\mu {\nabla}^s u) \nu& = g \quad \text{on } \Gamma_N, \\
    u& = 0 \quad \text{on } \Gamma_D.
\end{aligned}\right.
\end{equation}
From a physical point of view, this means that we consider an elastic test body that is fixed (zero displacement) at $\Gamma_D$ (Dirichlet condition) and subjected to a force $g$ on $\Gamma_N$ (Neumann condition). This results in the displacement $u$, which is measured on the boundary $\Gamma_N$.

The  weak formulation of the boundary value problem \eqref{direct2} reads:
\begin{equation}
    \int_\Omega \lambda \nabla \cdot u \nabla \cdot v +2\mu {\nabla}^s u : {\nabla}^s v  +\rho u\,dx = \int_{\Gamma_N} g \cdot v \, ds, \quad \forall v \in \mathcal V,
\end{equation}
where the function space is defined as:
\begin{equation*}
  \mathcal  V := \{ v \in H^1(\Omega)^d : v|_{\Gamma_D} = 0 \}.
\end{equation*}

We note that for $\lambda, \mu, \rho \in L_+^{\infty}(\Omega)$, the existence and uniqueness of a solution to the variational formulation (4) can be established using the Lax-Milgram theorem (see, e.g., [4]).

Measuring the boundary displacements resulting from applied forces on $\Gamma_N$ can be modeled by the Neumann-to-Dirichlet operator $\Lambda(\lambda, \mu,\rho)$, defined as:
\begin{equation*}
    \Lambda(\lambda, \mu,\rho) : L^2(\Gamma_N)^d \to L^2(\Gamma_N)^d, \quad g \mapsto u|_{\Gamma_N},
\end{equation*}
where $u \in \mathcal V$ solves \eqref{direct2}.

This operator is self-adjoint, compact, and linear. Its associated bilinear form is given by:
\begin{equation}
    \langle g, \Lambda(\lambda, \mu,\rho) h \rangle = \int_\Omega \lambda \nabla \cdot u^g \nabla \cdot u^h+2\mu {\nabla}^s u^g : {\nabla}^s u^h + \rho u^g\cdot u^h  \,dx,
\end{equation}
where $u^g$ solves problem \eqref{direct2} and $u^h$ solves the corresponding problem with boundary force $h \in L^2(\Gamma_N)^d$.
Next, we investigate the  differentiability of the Neumann-to-Dirichlet operator.
\begin{proposition}\label{prop_diff}
The Neumann-to-Dirichlet operator 
\[
\Lambda:  L^\infty_+(\Omega) \times L^\infty_+(\Omega) \times L^\infty_+(\Omega)  \longrightarrow   \mathcal{L}( L^2(\Gamma_{\textup N})^d),
\]
 is Fréchet  differentiable and its derivative in the direction   $ \tilde{\lambda}, \tilde{\mu},\tilde{\rho}$ 

is given by 
\begin{equation}
\label{8}
 \langle {\Lambda}'(\lambda, \mu, \rho) ( \tilde{\lambda}, \tilde{\mu},\tilde{\rho}) g , g \rangle = -\int_\Omega \Big( \tilde{\lambda} \  \nabla \cdot u^g \nabla \cdot u^g + 2 \ \tilde{\mu} \  \nabla^s u^g : \nabla^s u^g + \tilde{\rho} \  u^g \cdot u^g \Big) \ dx,
\end{equation}
where  $ u^g=u^g_{\lambda,\mu,\rho}$ is the solution of \eqref{direct2} for the boundary load $g\in  L^2(\Gamma_{\textup N})^d$. 
\end{proposition} 

\begin{proof} 
Introduce the Lagrangian 
\[
\begin{aligned}
L(\lambda,\mu,\rho,\varphi,\psi):= \int_\Omega \lambda(\nabla\cdot \varphi)\cdot(\nabla\cdot \varphi)+2\mu
 \nabla^s\varphi: \nabla^s\varphi +\rho\varphi\cdot\varphi\,dx\\
+\int_\Omega \lambda(\nabla\cdot \varphi)\cdot(\nabla\cdot \psi)+2\mu
 \nabla^s\varphi: \nabla^s\psi +\rho\varphi\cdot\psi\,dx-\int_{\Gamma_N}g\cdot\varphi\,dx.
\end{aligned}
\]
Obviously, we  have 
\[
L(\lambda,\mu,\rho,u^g,v)=  \langle {\Lambda}(\lambda, \mu, \rho)g,g\rangle.
\]
Since
\[
\sup_{\psi \in \mathcal V}L(\lambda,\mu,\rho,\varphi,\psi)= 
\left\{ \begin{array}{ccc}  \langle \Lambda(\lambda, \mu, \rho)g,g\rangle &\text{ if } \varphi=u^g_{\lambda,\mu,\rho},\\  +\infty  &\text{otherwise}, \end{array}\right.
\]
we get 
\[
\langle {\Lambda}(\lambda, \mu, \rho)g,g\rangle = \adjustlimits\inf_{\varphi \in \mathcal V}\sup_{\psi \in \mathcal V)}L(\lambda, \mu, \rho,\varphi,\psi).
\]
The functional $L$ is convex continuous with  respect to $\varphi$ and concave continuous with respect to $\psi$. 
According to Ekeland and Temam \cite{ekeland1974analyse},  the functional $L$ has a saddle point $(\bar u, \bar v)$  if and only if t $(\bar u, \bar v)$ solve the following system
\[
\begin{aligned}
\partial_\psi L(\lambda,\mu,\rho,\bar u,\bar v)\hat\varphi&=0,\\
\partial_\varphi L(\lambda,\mu,\rho,\bar u,\bar v)\hat \psi&=0,
\end{aligned}
\]
for all $\hat \varphi,  \hat \psi \in \mathcal V$. This yields that  $L$ has a saddle point $(\bar u,\bar v)=(u^g, -2u^g)$, where $u^g$ is the solution of \eqref{direct2}.

Let $\lambda_t=\lambda+t\tilde\lambda,  \mu_t=\mu+t\tilde\mu,   \rho_t=\rho+t\tilde\rho$.  Similarly to the above analysis, we  have 
\[
\langle {\Lambda}(\lambda_t, \mu_t, \rho_t)g,g\rangle = \adjustlimits\inf_{\varphi \in \mathcal V}\sup_{\psi \in \mathcal V)}L(\lambda_t, \mu_t, \rho_t,\varphi,\psi).
\]
Under the assumptions of Theorem \ref{CS} (which can be verified following the arguments as in \cite{meftahi2015sensitivity}), we obtain:
\[
\partial_t\langle {\Lambda}(\lambda_t, \mu_t, \rho_t)g,g\rangle\Big\vert_{t=0}=\partial_t L(\lambda_t, \mu_t, \rho_t,u^g,-2u^g)\Big\vert_{t=0}.
\]
This yields that  
\[
 \langle {\Lambda}'(\lambda, \mu, \rho) ( \tilde{\lambda}, \tilde{\mu},\tilde{\rho}) g , g \rangle = -\int_\Omega \Big( \tilde{\lambda} \  \nabla \cdot u^g \nabla \cdot u^g + 2 \ \tilde{\mu} \  \nabla^s u^g : \nabla^s u^g + \tilde{\rho} \  u^g \cdot u^g \Big) \ dx.
\]
 The directional derivative $ (\tilde{\lambda}, \tilde{\mu},\tilde{\rho})\mapsto \langle {\Lambda}'(\lambda, \mu, \rho) ( \tilde{\lambda}, \tilde{\mu},\tilde{\rho})$ is bounded and   linear.

Define the remainder 
\[
R(g):=  \langle {\Lambda}(\lambda+ \tilde{\lambda}, \mu+\tilde{\mu}, \rho+\tilde{\rho}) g , g \rangle
       -\langle {\Lambda}(\lambda, \mu, \rho) g , g \rangle
       - \langle {\Lambda}'(\lambda, \mu, \rho) ( \tilde{\lambda}, \tilde{\mu},\tilde{\rho}) g , g \rangle
\]
To prove the Fréchet differentiabililty, it suffices  to  show that 
\[
\sup_{g\in L^2(\Gamma_N)^d}\frac{\vert R(g) \vert}{ \Vert(\tilde \lambda,\tilde \mu,\tilde\rho)  \Vert}\rightarrow 0,\quad  \text{ as } \Vert(\tilde \lambda,\tilde \mu,\tilde\rho)  \Vert\rightarrow 0,
\]
where
\[
\Vert(\tilde \lambda,\tilde \mu,\tilde\rho)  \Vert:=\max\{\Vert \tilde\lambda \Vert_{L^\infty(\Omega)},
\Vert \tilde\mu \Vert_{L^\infty(\Omega)},\Vert \tilde\rho \Vert_{L^\infty(\Omega)}.
\]
Applying  Lemma~\ref{mono_loewn}  and Lemma~\ref{mono_bis} with $(\lambda_1,\mu_1,\rho_1):=(\lambda+\tilde \lambda,\mu+\tilde \mu, \rho+\tilde \rho)$  and $(\lambda_2,\mu_2,\rho_2):=(\lambda,\mu,\rho)$, 
where $u$ is  the solution  corresponding to  the parameters $(\lambda,\mu,\rho)$
and $\tilde u$ is the solution corresponding to the parameters  $(\lambda+\tilde \lambda,\mu+\tilde \mu,\rho+\tilde \rho)$,
we get that

\begin{equation*}
\begin{aligned}
 R(g)  &\geq \int_\Omega \left( (\lambda- \lambda -\tilde \lambda) \ \Vert \nabla \cdot u \Vert^2 + 2 \ (\mu-\mu- \tilde \mu) \  \Vert \nabla^s u \Vert^2_F  + (\rho-\rho-\tilde\rho) \ \Vert u \Vert^2 \right) \ dx\\
&+\int_\Omega \left(\tilde \lambda \Vert \nabla \cdot u \Vert^2 + 2 \tilde \mu \  \Vert \nabla^s u \Vert^2_F  + \tilde\rho \Vert u \Vert^2 \right) \ dx =0,
\end{aligned}
\end{equation*}
and 
\begin{equation*}
\begin{aligned}
 R(g)  \leq  &  \int_\Omega \Big( -\frac{\lambda \tilde\lambda}{\lambda+\tilde\lambda} \ \Vert \nabla \cdot u \Vert^2 - 2 \ \frac{\mu \tilde\mu}{\mu+\tilde\mu} \  \Vert \nabla^s  u \Vert^2_F - \frac{\rho \tilde\rho}{\rho+\tilde\rho}\ \Vert u^g \Vert^2  \\
& + \tilde\lambda \ \Vert \nabla \cdot  u\Vert^2+ 2 \tilde\mu \ \Vert \nabla^s \ u \Vert^2_F + \tilde\rho \ \Vert u \Vert^2 \Big) \ dx \\
& =  \int_\Omega \Big( \frac{\tilde\lambda^2}{\lambda+\tilde\lambda} \ \Vert \nabla \cdot u \Vert^2 + 2 \ \frac{ \tilde\mu^2}{\mu+\tilde\mu} \  \Vert \nabla^s \ u \Vert^2_F + \frac{\tilde\rho^2}{\rho+\tilde\rho}\ \Vert u \Vert^2 \Big) \ dx.
\end{aligned}
\end{equation*}
Therefore
\begin{equation*}
\frac{ R(g)}{\Vert(\tilde \lambda,\tilde \mu,\tilde\rho)  \Vert}
\leq   \int_\Omega \Big( \frac{\vert \tilde\lambda\vert}{\lambda+\tilde\lambda} \ \Vert \nabla \cdot u^g \Vert^2 + 2 \ \frac{ \vert \tilde \mu\vert}{\mu+\tilde\mu} \  \Vert \nabla^s \ u^g \Vert^2_F + \frac{\vert\tilde\rho\vert}{\rho+\tilde\rho}\ \Vert u^g \Vert^2 \Big) \ dx.
\end{equation*}
If $\Vert(\tilde \lambda,\tilde \mu,\tilde\rho)  \Vert$,   then $ \tilde \lambda, \tilde\mu, \tilde\rho \longrightarrow 0.$
Thanks to the  boundedness of $u$, we obtain 
\[
\frac{ R(g)}{\Vert(\tilde \lambda,\tilde \mu,\tilde\rho)  \Vert}\rightarrow 0 \text{ as }   \Vert(\tilde \lambda,\tilde \mu,\tilde\rho)  \Vert \rightarrow 0.
\]
This completes the proof.
\end{proof}
\subsection{Support definitions and the inverse  problem}
\begin{definition}
A relatively open set  $U\subseteq \overline{\Omega}$ is called connected to $\partial\Omega$ if  $U \cap\Omega$
is connected and $U\cap \partial \Omega \neq \emptyset$.
\end{definition}
\begin{definition}
 For a measurable function  $\varphi:  \Omega \rightarrow \mathbb{R}$,  we define
 \begin{itemize}
 \item[a)] the support supp$(\varphi)$ as the complement (in $\overline{\Omega}$) of the union of those relatively
 open $U\subseteq\overline{\Omega}$, for which  $\varphi|_{U}\equiv 0$.
 \item[b)] the inner support inn supp$(\varphi)$ as the union of those open sets $U\subseteq \Omega$, for which  ess $\inf_{x\in U}|\varphi(x)| >0$.
 \item[c)] the outer support $\text{out}_{\partial\Omega}$ supp$(\varphi)$ as the complement (in $\overline{\Omega}$) of the union of those relatively open $U\subseteq \overline{\Omega}$ 
 that are connected to $\partial\Omega$ and for which 
 $\varphi|_{U}\equiv 0$.
 \end{itemize}
\end{definition}
\begin{lemma}
If $\varphi$ is piecewise continuous,  $\text{supp}(\varphi) \subset  \Omega$, and $\Omega \setminus  \text{supp}(\varphi)$ is
connected, then
\[
\overline{\text{innsupp}(\varphi)} = \text{supp}(\varphi) = \text{out}_{\partial\Omega} \text{supp}(\varphi).
\]
\end{lemma}
In this work, we are specifically interested in identifying the spatial distribution (or shape) of perturbations in the material parameters $\lambda$, $\mu$, and $\rho$, assumed to vary over a bounded domain $\Omega \subset \mathbb{R}^d$. These parameters represent  deviations from a homogeneous background medium characterized by the constant reference values $\lambda_0$, $\mu_0$, and $\rho_0 > 0$. More precisely, we consider:
\begin{equation}
\lambda(x) = \lambda_0 + \delta\lambda(x), \quad
\mu(x) = \mu_0 + \delta\mu(x), \quad
\rho(x) = \rho_0 + \delta\rho(x), \quad x \in \Omega,
\end{equation}
where the  perturbations $\delta\lambda$, $\delta\mu$, and $\delta\rho$ are to be determined.

The inverse problem we address here is the following.
\begin{equation}\label{invp2}
\tag{IP$$}\left\{
\begin{aligned}
&\text{Determine  the support of  perturbations  }\; \text{supp  } \delta \lambda(x)\cup \text{supp  } \delta \mu(x) \cup
\text{supp  }\delta \rho(x)\\ 
&\text{ from the Neumann-to-Dirichlet map } \Lambda(\lambda, \mu, \rho).
\end{aligned}\right.
\end{equation}
\section{Shape  reconstruction  via the linearized  monotonicity  test}
In this this section, we assume that parameters  variations   are of the form
$\delta\lambda=(\lambda_1-\lambda_0)\chi_{D}$,   $\delta\mu=(\mu_1-\mu_0)\chi_{D}$ and
$\delta\rho=(\rho_1-\rho_0)\chi_{D}$,
and    the values  $\lambda_0, \lambda_1,  \mu_0, \mu_1$, $\rho_0$ and $\rho_1$  are known {\it a priori} but the information about the geometry  $D\Subset\Omega$ is missing. For the full characterization of ($\delta\lambda,\delta\mu,\delta\rho)$   it is then sufficient to find the geometry $D$.

 We begin by introducing the concept of Loewner monotonicity, followed by the derivation of the standard (non-linearized) monotonicity method and its more efficient linearized variant. Finally, we present a noise-robust regularization approach based on monotonicity.
\subsection{Loewner monotonicity}
To introduce the monotonicity-based method, we first recall the following lemma.
\begin{lemma}\label{mono_loewn}
Let $\lambda_1, \lambda_2, \mu_1, \mu_2, \rho_1, \rho_2  \in L^\infty_+(\Omega)$,  and let  $g\in L^2(\Gamma_{\textup N})^d$ be an applied boundary load. The corresponding solutions of (\ref{direct2}) are denoted by $u_1:=u^g_{\lambda_1,\mu_1,\rho_1},\ u_2:=u^{g}_{\lambda_2,\mu_2,\rho_2}\in \mathcal{V}$. Then
\begin{equation}
\begin{aligned}
&\int_\Omega(\lambda_1-\lambda_2)\Vert\nabla\cdot u_2 \Vert^2\,dx+2\int_\Omega(\mu_1-\mu_2)\Vert\nabla^s u_2\Vert^2_F\,dx + \int_\Omega (\rho_1-\rho_2)\Vert u_2\Vert^2\,dx\\
&\geq \langle g,\Lambda(\lambda_2,\mu_2,\rho_2)g\rangle-\langle g,\Lambda(\lambda_1,\mu_1,\rho_1)g\rangle \geq\\
&\int_\Omega(\lambda_1-\lambda_2)\Vert\nabla\cdot u_1 \Vert^2\,dx+ 2\int_\Omega(\mu_1-\mu_2)\Vert\nabla^s u_1\Vert^2_F\,dx +  \int_\Omega(\rho_1-\rho_2)\Vert u_1\Vert^2\,dx.
\end{aligned}
\end{equation}
\end{lemma}
\begin{proof}
The proof of Lemma \ref{mono_loewn} is given in \cite{meftahi2025stability}.
\end{proof}
\begin{lemma}
  \label{mono_bis}
Let $\lambda_1, \lambda_2, \mu_1, \mu_2, \rho_1, \rho_2  \in L^\infty_+(\Omega)$,  and let  $g\in L^2(\Gamma_{\textup N},\R^d)$ be an applied boundary load. The corresponding solutions of (\ref{direct2}) are denoted by $u_1:=u^g_{\lambda_1,\mu_1,\rho_1},\ u_2:=u^{g}_{\lambda_2,\mu_2,\rho_2
      }\in \mathcal{V}$. Then
\begin{equation*}
\begin{aligned}
& \langle g,\Lambda(\lambda_2,\mu_2,\rho_2)g\rangle-\langle g,\Lambda(\lambda_1,\mu_1,\rho_1)g\rangle \geq\\
&\int_\Omega(\lambda_2-\frac{\lambda_2^2}{\lambda_1})\Vert\nabla\cdot u_1 \Vert^2\,dx+ 2\int_\Omega(\mu_2-\frac{\mu_2^2}{\mu_1})\Vert\nabla^s u_1\Vert^2_F\,dx +  \int_\Omega(\rho_2-\frac{\rho_2^2}{\rho_1})\Vert u_1\Vert^2\,dx. \\
& = \int_\Omega \frac{\lambda_2}{\lambda_1} (\lambda_1-\lambda_2) \Vert\nabla\cdot u_1 \Vert^2\,dx+ 2\int_\Omega \frac{\mu_2}{\mu_1}  (\mu_1-\mu_2)\Vert\nabla^s u_1\Vert^2_F\,dx +  \int_\Omega \frac{\rho_2}{\rho_1} (\rho_1-\rho_2)\Vert u_1\Vert^2\,dx.
\end{aligned}
\end{equation*}
\end{lemma}
\begin{proof} 
We  have 
\begin{equation*}
\begin{aligned}
&  \langle g,\Lambda(\lambda_2,\mu_2,\rho_2)g\rangle-\langle g,\Lambda(\lambda_1,\mu_1,\rho_1)g\rangle 
 =  \int_\Omega \Big( \lambda_2 \  \nabla \cdot (u_2-u_1) \  \nabla \cdot (u_2-u_1) \\
& + 2  \ \mu_2 \ \nabla^s (u_2-u_1) : \nabla^s (u_2-u_1) + \rho_2 \ (u_2-u_1) \cdot (u_2-u_1)  \\
&    + (\lambda_1-\lambda_2) \  \Vert\nabla\cdot u_1 \Vert^2 + 2 \ ( \mu_1-\mu_2) \ \Vert\nabla^s u_1\Vert^2_F + (\rho_1-\rho_2) \ \Vert u_1\Vert^2 \Big) \  dx.  
\end{aligned}
\end{equation*}
Therefore
\begin{equation*}
\begin{aligned}
&  \langle g,\Lambda(\lambda_2,\mu_2,\rho_2)g\rangle-\langle g,\Lambda(\lambda_1,\mu_1,\rho_1)g\rangle \\
& = \int_\Omega \Big( \lambda_2  \   \Vert\nabla\cdot u_2 \Vert^2 + \lambda_2 \  \Vert\nabla\cdot u_1 \Vert^2
-2 \ \lambda_2 \ \nabla \cdot u_2 \ \nabla \cdot u_1+ \lambda_1  \  \Vert\nabla\cdot u_1 \Vert^2 \\
&- \lambda_2 \  \Vert\nabla\cdot u_1 \Vert^2 + 2 \  \mu_2 \ \Vert\nabla^s u_2\Vert^2_F 
+ 2 \ \mu_2 \  \Vert\nabla^s u_1\Vert^2_F - 4 \ \mu_2 \ \nabla^s u_2 : \nabla^s u_1 \\
&+  2 \ \mu_1 \ \Vert\nabla^s u_1\Vert^2_F - 2 \ \mu_2 \ \Vert\nabla^s u_1\Vert^2_F \\
&  + \rho_2 \   \Vert u_2\Vert^2 + \rho_2 \  \Vert u_1\Vert^2  -2 \ \rho_2 \ u_2 \cdot u_1+ \rho_1 \  \Vert u_1\Vert^2  - \rho_2 \  \Vert u_1\Vert^2  \Big) \ dx \\
& = \int_\Omega \Big( \lambda_2 \  \Vert\nabla\cdot u_2 \Vert^2 - 2 \ \lambda_2 \ \nabla \cdot u_2 \  \nabla \cdot u_1+ \lambda_1 \  \Vert\nabla\cdot u_1 \Vert^2   + 2 \ \mu_2 \   \Vert\nabla^s u_2\Vert^2_F \\
&- 4  \ \mu_2 \ \nabla^s \ u_2 : \nabla^s  \ u_1  + 2 \ \mu_1 \  \Vert\nabla^s u_1\Vert^2_F  + \rho_2 \    \Vert u_2\Vert^2  - 2 \ \rho_2 \ u_2 \cdot u_1 + \rho_1 \   \Vert u_1\Vert^2 \Big) \ dx \\
& = \int_\Omega \Big( \lambda_1 \ \Vert \nabla \cdot u_1 - \frac{\lambda_2}{\lambda_1} \ \nabla \cdot u_2 \Vert^2 + 2 \ \mu_1 \Vert \nabla^s \ u_1 - \frac{\mu_2}{\mu_1} \ \nabla^s \ u_2 \Vert^2_F+ \rho_1 \ \Vert u_1 - \frac{\rho_2}{\rho_1} \ u_2 \Vert^2 \\
& + ( \lambda_2 - \frac{\lambda_2^2}{\lambda_1}) \ \Vert \nabla \cdot u_2 \Vert^2 + 2 \ (\mu_2 -\frac{\mu_2^2}{\mu_1} ) \Vert \nabla^s \ u_2 \Vert^2_F+ (\rho_2-\frac{\rho_2^2}{\rho_1}) \Vert u_2 \Vert^2 \Big) \  dx \\
& \geq  \int_\Omega \Big(  ( \lambda_2 - \frac{\lambda_2^2}{\lambda_1}) \ \Vert \nabla \cdot u_1 \Vert^2 + 2 \ (\mu_2 -\frac{\mu_2^2}{\mu_1} ) \Vert \nabla^s \ u_1 \Vert^2_F+ (\rho_2-\frac{\rho_2^2}{\rho_1}) \Vert u_1 \Vert^2 \Big) \  dx. 
\end{aligned}
\end{equation*}
\end{proof}
 Lemma \ref{mono_loewn}  implies monotonicity of the mapping  $(\lambda,\mu,\rho)\rightarrow \Lambda(\lambda,\mu,\rho)$ with respect to the following partial orderings. For symmetric operators $T,  S \in  \mathcal{L}(L^2(\Gamma_N)^d)$ we introduce the
 semidefinite (aka Loewner) ordering:
\[
T\preceq S \text{ denotes that } \int_{\Gamma_N}g(S-T)g\,ds \geq 0,\quad\text{ for all  }
g\in L^2(\Gamma_N)^d.
\]
Also,  for functions  $\kappa, \tau \in L^\infty(\Omega)$
\[
\kappa  \leq \tau \text{ denotes that }  \tau(x)\geq \kappa(x) \text{  a.e } x\in \Omega.
\]
Using this notation, we directly derive the following result as a consequence of Lemma \ref{mono_loewn}.
\begin{corollary}
Let $\lambda_0, \lambda_1, \mu_0, 
\mu_1,   \rho_0, \rho_1\in L^{\infty}(\Omega)$. Then
\begin{equation}
 \lambda_0 \leq \lambda_1,   \quad
\mu_0 \leq \mu_1 \quad\text{ and }\quad \rho_0 \leq \rho_1, \quad\text{ implies } 
\quad\Lambda(\lambda_0, \mu_0,\rho_1) \succeq \Lambda(\lambda_1, \mu_1,\rho_1).
\end{equation}
\end{corollary}

\begin{corollary}
Let $\lambda_0, \lambda_1, \mu_0, 
\mu_1,   \rho_0, \rho_1\in L^{\infty}(\Omega)$. Then
\begin{equation}
 \lambda_0 \leq \lambda_1,   
\mu_0 \leq \mu_1 \text{ and } \rho_0 \leq \rho_1, \text{ implies } 
\Lambda'(\lambda, \mu,\rho)(\lambda_0, \mu_0,\rho_1) \succeq \Lambda'(\lambda, \mu,\rho)(\lambda_1, \mu_1,\rho_1).
\end{equation}
\end{corollary}
\subsection{The linearized monotonicity  test}
The shape $D$ can be seen as the outer support of of  $(\lambda-\lambda_0,\mu-\mu_0,\rho-\rho_0)^T$  i.e.,
\[
D= {\rm out}_{\partial\Omega} {\rm supp}(\lambda-\lambda_0,\mu-\mu_0,\rho-\rho_0)^T.
\]
Under the assumption, that the complement of $D$ is connected to
 the boundary $\partial\Omega$, and  the parameters  $\lambda_0, \mu_0, \rho_0$  in  the
 background and   $\lambda_1, \mu_1,\rho_1$ inside $D$ are constant,  we  have the following results:
\begin{theorem}
\label{mono_test}
Let $\lambda_0,  \lambda_1, \mu_0, \mu_1, \rho_0, \rho_1\in \mathbb{R}_+$  
with  $ \lambda_1 \geq \lambda_0, \ \mu_1 \geq \mu_0, \ \rho _1 \geq \rho_0$  and assume that 
\[
\lambda=\lambda_0+(\lambda_1-\lambda_0)\chi_D, \quad  \mu=\mu_0+(\mu_1-\mu_0)\chi_D,\quad 
\rho=\rho_0+(\rho_1-\rho_0)\chi_D.
\]
Further on let $C^{\lambda}, C^{\mu}, C^{\rho}\in \mathbb{R}_+$ and   $C^{\lambda}+ C^{\mu}+ C^{\rho}>0$,  with 
\[
C^{\lambda}\leq \frac{\lambda_0}{\lambda_1}(\lambda_1-\lambda_0),\quad
C^{\mu}\leq \frac{\mu_0}{\mu_1}(\mu_1-\mu_0),\quad
C^{\rho}\leq \frac{\rho_0}{\rho_1}(\rho_1-\rho_0).
\]
Then, for every open set  $B$  
\[
\begin{aligned}
&B\subseteq D \text{ if and only if }\\
&\Lambda(\lambda_0,\mu_0,\rho_0)- \Lambda(\lambda,\mu,\rho)+    \Lambda'(\lambda_0,\mu_0,\rho_0)(C^{\lambda}\chi_B, C^{\mu}\chi_B,
C^{\rho}\chi_B )\geq 0 .
\end{aligned}
\]
\end{theorem}
\begin{theorem}
\label{mono_test_noise}
Let $\lambda_0,  \lambda_1, \mu_0, \mu_1, \rho_0, \rho_1\in \mathbb{R}_+$  
with  $ \lambda_1 \geq \lambda_0, \ \mu_1 \geq \mu_0, \ \rho _1 \geq \rho_0$  and assume that 
\[
\lambda=\lambda_0+(\lambda_1-\lambda_0)\chi_D, \quad  \mu=\mu_0+(\mu_1-\mu_0)\chi_D,\quad 
\rho=\rho_0+(\rho_1-\rho_0)\chi_D.
\]
Further on let $C^{\lambda}, C^{\mu}, C^{\rho}\in \mathbb{R}_+$ and   $C^{\lambda}+ C^{\mu}+ C^{\rho}>0$,  with 
\[
C^{\lambda}\leq \frac{\lambda_0}{\lambda_1}(\lambda_1-\lambda_0),\quad
C^{\mu}\leq \frac{\mu_0}{\mu_1}(\mu_1-\mu_0),\quad
C^{\rho}\leq \frac{\rho_0}{\rho_1}(\rho_1-\rho_0).
\]
Let  $\Lambda^{\delta}$ be the Neumann-to-Dirichlet operator   for noisy difference measurements with noise level
 $\delta>0$.

 Then, for every open set  $B$   there exists a noise level  $\bar \delta$ such that for all $0<\delta<\bar \delta$
\[
\begin{aligned}
&B\subseteq D \text{ if and only if }\\
&\Lambda(\lambda_0,\mu_0,\rho_0)-\Lambda^{\delta}(\lambda,\mu,\rho)+    \Lambda'(\lambda_0,\mu_0,\rho_0)(C^{\lambda}\chi_B, C^{\mu}\chi_B,
C^{\rho}\chi_B )+\delta I\geq 0.
\end{aligned}
\]
\end{theorem}
The proofs of Theorems~\ref{mono_test} and~\ref{mono_test_noise} follow by arguments analogous to those used in Corollaries 2 and 3 of~\cite{eberle2022monotonicity}.
\subsection{Numerical results using the linearized monotonicity test}\label{sec_mono_test}
In this subsection, we assume that the parameters $\lambda_0, \lambda_1, \mu_0,\mu_1, \rho_0$ and $\rho_1$
are known and we aim to reconstruct only the shape $D$.
 
We denote  by $\bar T\in \mathbb{R}^{m\times m}$ the discrete version of an operator $T\in \mathcal L (L^2( \Gamma_N)^d)$, i.e., its Galerkin projection to the span of $g_1,\ldots,g_m$.

 Following Theorem \ref{mono_test}, for each test ball $B$, we computed the eigenvalues of
\[
\overline \Lambda(\lambda_0,\mu_0,\rho_0)- \overline\Lambda(\lambda,\mu,\rho)+ \overline{\Lambda'}(\lambda_0,\mu_0,\rho_o)(C^{\lambda}\chi_B, C^{\mu}\chi_B,C^{\rho}\chi_B).
\]
We also considered the noisy data case 
\[
\overline\Lambda^\delta:=\overline \Lambda(\lambda,\mu,\rho)- \overline\Lambda(\lambda_0,\mu_0,\rho_0) + \delta \Vert \overline \Lambda(\lambda,\mu,\rho)- \overline\Lambda(\lambda_0,\mu_0,\rho_0)\Vert_F
\frac{E}{\Vert E\Vert_F},
\]
where the entries of $E\in \mathbb{R}^{m\times m}$ are normally distributed random variables with a mean of zero and a standard deviation of one.  
In the noisy data case, we mark those test balls for which all eigenvalues of 
\[
 \overline{\Lambda'}(\lambda_0,\mu_0,\rho_o)(C^{\lambda}\chi_B, C^{\mu}\chi_B,C^{\rho}\chi_B) -\overline\Lambda^\delta+ \delta I
\]
are positive.

 We summarize the above results in following  steps:
\texttt{
\begin{itemize}
\item[(i)] For each  test ball $B$ :
\begin{itemize}
\item[a.] for noiseless data, compute the eigenvalues of 
\begin{equation}\label{test_int1}
\overline \Lambda(\lambda_0,\mu_0,\rho_0)- \overline\Lambda(\lambda,\mu,\rho)+ \overline{\Lambda'}(\lambda_0,\mu_0,\rho_o)(C^{\lambda}\chi_B, C^{\mu}\chi_B,C^{\rho}\chi_B).
\end{equation}
\item[b.] for noisy data, compute the eigenvalues of 
\begin{equation}\label{test_int2}
 \overline{\Lambda'}(\lambda_0,\mu_0,\rho_o)(C^{\lambda}\chi_B, C^{\mu}\chi_B,C^{\rho}\chi_B) -\overline\Lambda^\delta+ \delta I.
\end{equation}
\end{itemize}
\item[(ii)]  If all the eigenvalues are positive, mark the ball $B$ as inside the domain $D$.
\end{itemize}
}
For the following numerical examples, we use a FEM mesh of the geometry $\Omega=[0,1]\times [0,1]$ with $5248$ elements.   We define the Neumann boundary 
\[ \Gamma_N=\{x=0\}\cup \{x=1\}\cup\{y=0\}, \]
as the union of patches $\Gamma_N^{(l)}$, for $l = 1, \dots, m$, which are assumed to be relatively open and connected, such that  
\[
\Gamma_N = \bigcup_{l=1}^{m} \Gamma_N^{(l)}, \quad \Gamma_N^{(i)} \cap \Gamma_N^{(j)} = \emptyset \text{ for } i \neq j.
\]  
We consider Neumann boundary data \( g_l \), for \( l = 1, \dots, m \), where each \( g_l \) is applied to a distinct patch \( \Gamma_N^{(l)} \). The value of \( g_l \) represents the incoming normal component on the corresponding patch.

In our setup, we have used 100 test balls and performed \( m = 19 \) measurements.  The test  parameters were set as $C^{\lambda}=C^{\mu}=C^{\rho}=0.5$.  The integrals appearing in equations \eqref{test_int1} and \eqref{test_int2} are computed using a MATLAB-based quadrature scheme.
 \begin{figure}[H]
\begin{center}
\subfloat{\includegraphics[scale=0.4]{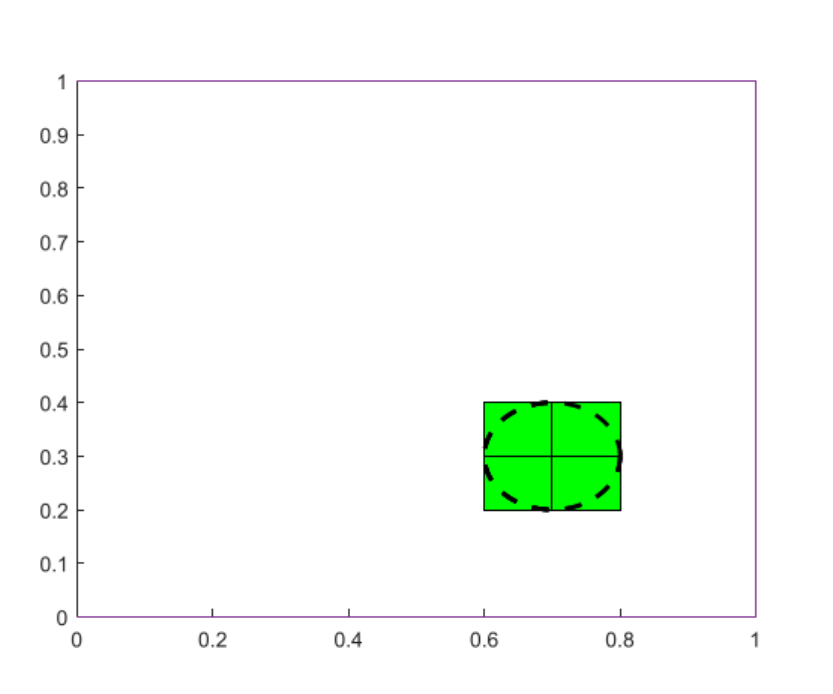}}
\subfloat{\includegraphics[scale=0.4]{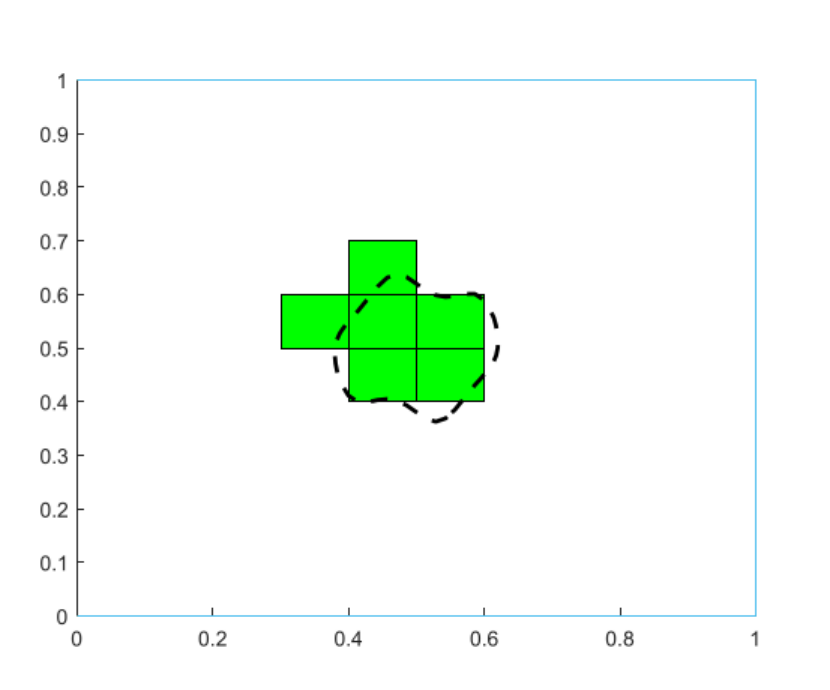}}
\subfloat{\includegraphics[scale=0.4]{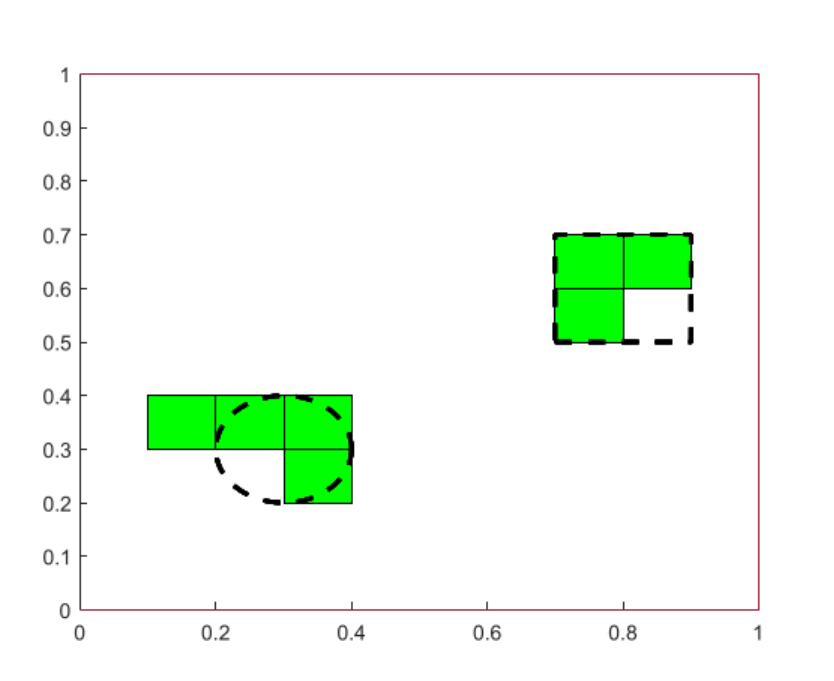}}\\
\subfloat{\includegraphics[scale=0.4]{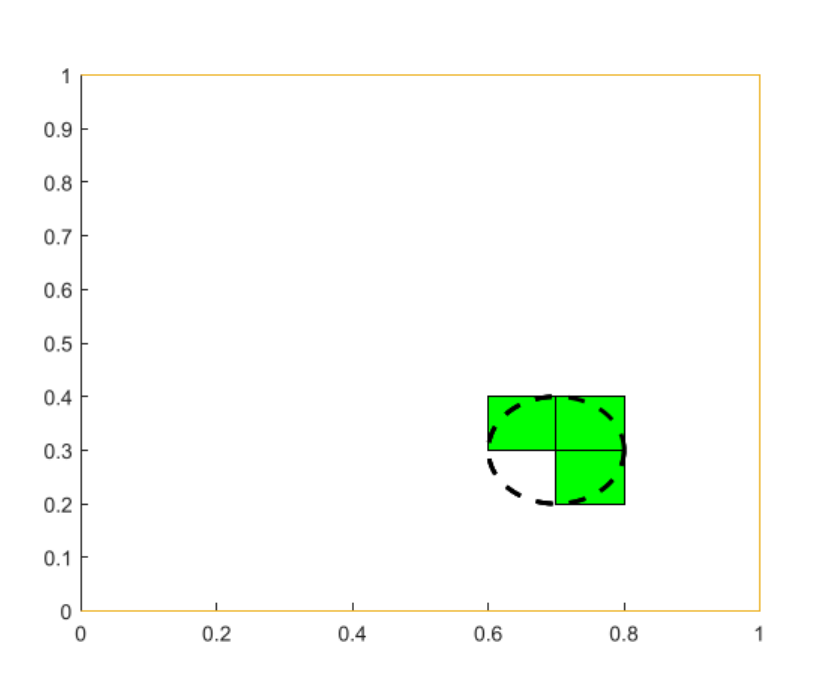}}
\subfloat{\includegraphics[scale=0.4]{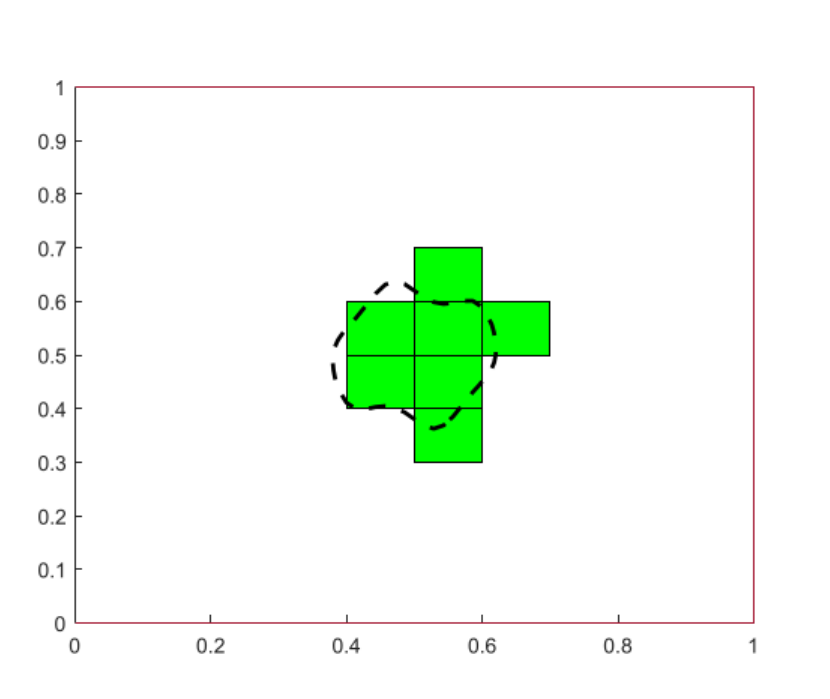}}
\subfloat{\includegraphics[scale=0.4]{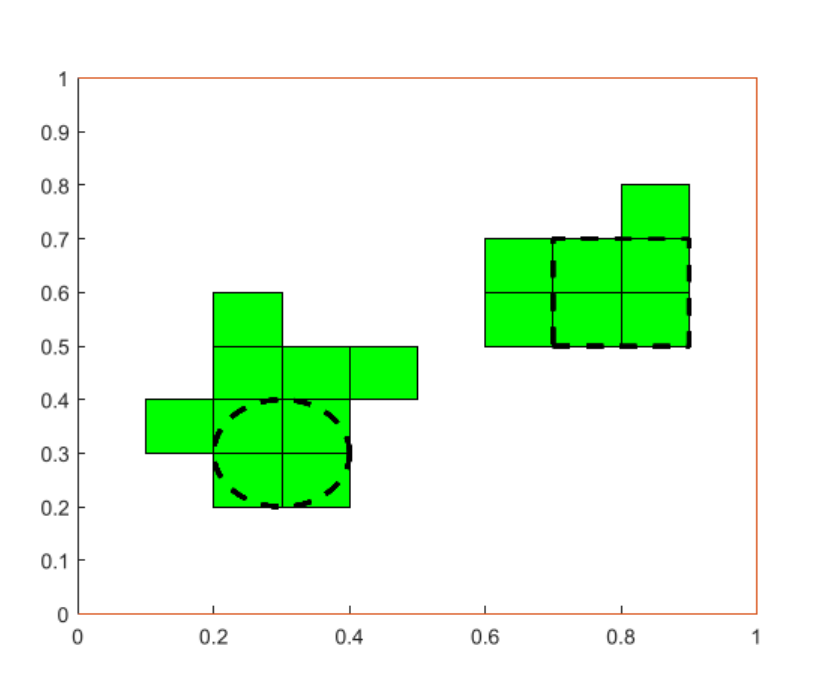}}
\caption{Reconstruction of the shape $D$ using Algorithm 1. The first row presents the reconstruction from noise-free data, while the second row shows the reconstruction from noisy data with a noise level of 
$\delta=0.001$ .}
\end{center}
\end{figure}
\section{Reconstruction subject to   monotonicity constraints}
The numerical experiments based on the linearized monotonicity test exhibit significant sensitivity to noise. To enhance the robustness of the reconstruction, we employ a regularization method incorporating monotonicity constraints, which will be detailed in the following subsections.
\subsection{Standard one-step linearization techniques}
In this subection we   use one-step linearization techniques   to reconstruct the  shape
$D= {\rm out}_{\partial\Omega} {\rm supp}(\lambda-\lambda_0,\mu-\mu_0,\rho-\rho_0)^T$.
  Moreprecisely,  we compare the matrix of the discretized Neumann-to-Dirichlet operator $\Lambda(\lambda, \mu, \rho)$ with $\Lambda(\lambda_0, \mu_0, \rho_0)$ for some reference parameters $(\lambda_0, \mu_0, \rho_0)$:
\begin{equation}\label{eq_linear}
 \Lambda(\lambda, \mu, \rho) - \Lambda(\lambda_0, \mu_0,\rho_0)\approx \Lambda'(\lambda_0, \mu_0,\rho_0)(\lambda-\lambda_0, \mu-\mu_0,\rho-\rho_0).
\end{equation}
where $\Lambda'(\lambda_0, \mu_0,\rho_0)$ is the Fr\'{e}chet derivative of the Neumann-to-Dirichlet operator $\Lambda$.\\

To recover numerically the shape $D$  from equation\eqref{eq_linear},  we discretize the domain $\Omega$  into $L$ disjoint pixels $B_k$, where each $B_k$ is assumed to be open and $\Omega \setminus B_k$ is connected and $B_k \cap B_i = \emptyset$  for $i \neq k.$  We make the piecewise constant ansatz for $(\alpha, \beta, \gamma)$ as:
\[
\alpha(x) = \sum_{k=1}^L \alpha_k \chi_{B_k}(x), \quad
\beta(x) = \sum_{k=1}^L \beta_k \chi_{B_k}(x), \quad
\gamma(x)  = \sum_{k=1}^L \gamma_k \chi_{B_k}(x),
\]
where $\chi_{B_k}$ is the characteristic function of the pixel $B_k$, and $\alpha, \beta$ and $\gamma$  approximate $(\lambda-\lambda_0, \mu-\mu_0,\rho-\rho_0)$. Let  $ g_1,\ldots, g_m\in L^2(\Gamma_N)^d$
  denote $m$ orthonormal boundary loads.  We denote  $T^{\lambda}_k,  T^{\mu}_k,  T^{\rho}_k$  and $U$ the matrices with entries 
\[
(T^{\lambda}_k)_{i,j}:=\int_{B_k}\nabla \cdot u^{g_i}_0 \nabla \cdot u^{g_j}_0 \,dx,  \quad
(T^{\mu}_k)_{i,j}:= 2\int_{B_k} \nabla  u^{g_i}_0 :\nabla  u^{g_j}_0 \,dx,\quad
(T^{\rho}_k)_{i,j}:=\int_{B_k}   u^{g_i}_0 \cdot u^{g_j}_0 \,dx,
\]
\[
U_{i,j}=\int_{\Gamma_N}g_i\cdot (u^{g_j}_0- u^{g_j})\,ds,
\]
where  $u^{g_k}_0$ resp. $u^{g_k}$   is the solution to the boundary value problem \eqref{direct2} with respect to the parameters $\lambda_0, \mu_0, \rho_0$  resp.  the parameters $\lambda, \mu, \rho$ and  the applied force $g_k$.

This approach yields the linear equation system:
\begin{equation}\label{mat_sys}
\sum_{k=1}^L \alpha_k T^{\lambda}_k+\beta_k T^{\mu}_k+\gamma_k T^{\rho}_k=U.
\end{equation}
We can  rewrite \eqref{mat_sys} in the following  vectorial form

\begin{equation}\label{vec_sys}
 \mathcal{A}\alpha+ \mathcal{B} \beta + \mathcal{C}\gamma= \mathcal{D}.
\end{equation}
where
\[
\mathcal{A}= (\mathcal{A}_{i,j})\in \mathbb{R}^{m^2\times L},\quad \mathcal{B}= (\mathcal{B}_{i,j})\in \mathbb{R}^{m^2\times L}, \quad\mathcal{C}= (\mathcal{C}_{i,j})\in \mathbb{R}^{m^2\times L},    \quad \mathcal{D} =(\mathcal{D}_i)^{m^2}_{i=1}\in \mathbb{R}^{m^2}, 
\]
with
\[
\mathcal{A}_{(i-1)m+j,k}=(T^{\lambda}_k)_{i,j},\quad \mathcal{B}_{(i-1)m+j,k}=(T^{\mu}_k)_{i,j},\quad
\mathcal{C}_{(i-1)m+j,k}=(T^{\rho}_k)_{i,j},\quad \mathcal{D}_{(i-1)m+j}= U_{i,j}.
\]
To estimate the parameters $\alpha, \beta, \gamma$, a standard approach is to consider the following minimization problems:
\begin{equation}\label{min_prob}
\left\Vert \mathcal{A} \,|\, \mathcal{B} \,|\, \mathcal{C}
\begin{pmatrix} \alpha \\ \beta \\ \gamma \end{pmatrix} - \mathcal{D} \right\Vert^2_2 \rightarrow \min !
\end{equation}
and
\begin{equation}\label{min_fro}
\left\Vert R(\alpha,\beta,\gamma) \right\Vert^2_F \rightarrow \min !
\end{equation}
where
\[
R(\alpha,\beta,\gamma) := \sum_{k=1}^L \alpha_k T^{\lambda}_k + \beta_k T^{\mu}_k + \gamma_k T^{\rho}_k - U.
\]
The notation $\|\cdot\|_F$ denotes the Frobenius norm.

However, the above  minimization problems  are  ill-posed, and therefore a regularization strategy is required to obtain a stable and meaningful solution.
In the following, we introduce various regularization techniques to address the minimization problem \eqref{min_fro}.
\subsection{Regularization via monotonicity constraints}
  To compute a stable   resolution of problem  \eqref{min_fro},  we enforce a monotonicity constraint  by imposing    a box constraints  on the parameters   as follows:
 we assume first   that  the variations are bounded
\[
  \lambda_{\min}\leq  \lambda_1-\lambda_0\leq \lambda_{\max},\quad   \mu_{\min}\leq  \mu_1-\mu_0\leq \mu_{\max},\quad
 \rho_{\min}\leq  \rho_1-\rho_0\leq \rho_{\max},
\]
 Let  
\[
 a_{\max} =  \lambda_0-\frac{\lambda_0^2}{\lambda_0+\lambda_{\min}},\quad    b_{\max} =  \mu_0-\frac{\mu_0^2}{\mu_0+\mu_{\min}}, \quad c_{\max} =  \rho_0-\frac{\rho_0^2}{\rho_0+\rho_{\min}},
\]
and define 
\[
\tau_1= \frac{b_{\max}}{a_{\max}}, \quad   \tau_2= \frac{c_{\max}}{a_{\max}}. 
\]
We   have 
\[
 \lambda_0-\frac{\lambda_0^2}{\lambda_1}\geq a, \quad
 \mu_0-\frac{\mu_0^2}{\mu_1}\geq \tau_1 a,  \quad
\rho_0-\frac{\rho_0^2}{\rho_1} \geq \tau_2 a,  \quad \forall~ 0\leq a\leq a_{\max}.
\]
To enforce the monotonicity constraint, we  rewrite the minimization problem  as follows
\begin{equation}\label{min_prob}
 \min_{\zeta \in \mathcal{C}}\Vert R(\zeta,\tau_1\zeta,\tau_2\zeta) \Vert^2_F,
\end{equation}
where the set of admissible parameters is given by:
\[
\mathcal{C}:=\left\{  \zeta \in L^{\infty}(\Omega):  \zeta=\sum_{k=1}^L \zeta_k\chi_{B_k},  \quad  \zeta_k\in \mathbb{R}, \quad 0\leq\zeta_k\leq \min(a_{\max},\beta_k)\right\},
\]
with  
\[
\beta_k:=\max\{ a>0:  \quad  U \geq  aT^{\lambda}_k+\tau_1a T^{\mu}_k+ \tau_2a T^{\rho}_k \}.
\]
According to \cite{eberle2022monotonicity}, one can demonstrate that
\[
\beta_k =  - \frac{1}{\Theta_{\min}\left (L^{-1} T_k (L^*)^{-1}\right)},
\]
where  $\Theta_{\min}(.)$ compute the most negative value,  $ L L^*$
 is the Cholesky factorization of the  matrix $U$  and  $T_k=  T^{\lambda}_k+\tau_1 T^{\mu}_k+ \tau_2 T^{\rho}_k$.
\begin{theorem}\label{thm_min_prob}
The   minimization problem  \eqref{min_prob} has a    solution $ \hat\zeta \in \mathcal{C}$. 
\end{theorem}
\begin{proof}
The objective function
\[
J(\zeta) := \left\| \sum_{k=1}^L \zeta_k T_k - U \right\|_F^2
\]
is continuous with respect the finite-dimensional vector \( \zeta \in \mathbb{R}^L \), since it is a quadratic function.

The admissible set \( \mathcal{C} \subset \mathbb{R}^L \) is defined by box constraints on each coefficient \( \zeta_k \) is  compact.
 Hence, the minimization problem admits at least one solution \( \hat{\zeta} \in \mathcal{C} \).

\end{proof}

For  noisy data $U^{\delta}$,  with
\[
\Vert U^\delta-U \Vert_F\leq \delta,
\]
we consider the minimization problem
\begin{equation}\label{min_prob_noise}
 \min_{\zeta \in \mathcal{C}^{\delta}}\Vert R^{\delta}(\zeta,\tau_1\zeta,\tau_2\zeta) \Vert^2_F,
\end{equation}
 where
\[
R^{\delta}(\zeta,\tau_1\zeta,\tau_2\zeta) := \sum_{k=1}^L \zeta_k T^{\lambda}_k+\tau_1\zeta T^{\mu}_k+\tau_2 \zeta T^{\rho}_k-U^{\delta},  
\]
with  
\[
\mathcal{C}^{\delta}:=\left\{  \zeta \in L^{\infty}(\Omega):  \zeta=\sum_{k=1}^L \zeta_k\chi_{B_k},  \quad  \zeta_k\in \mathbb{R}, \quad 0\leq\zeta_k\leq \min(a_{\max},\beta_k)\right\},
\]
and
\[
\beta^{\delta}_k:=\max\{ a>0:  \quad  U^{\delta}  +\delta I \geq  aT^{\lambda}_k+\tau_1a T^{\mu}_k+ \tau_2a T^{\rho}_k \}.
\]
Building on the results of \cite{eberle2022monotonicity}, we can establish that
\[
\beta^{\delta}_k =  - \frac{1}{\Theta_{\min}\left (L^{-1}_\delta T_k (L^*_\delta)^{-1}\right)},
\]
where   $ L_\delta L^*_{\delta}$
 is the Cholesky factorization of the  matrix $U^\delta+\delta I$.
\begin{theorem}\label{thm_min_prob_noise}
The minimization problem \eqref{min_prob_noise} admits a solution \( \hat\zeta^{\delta} \in \mathcal{C}^{\delta} \).
Let  $\hat \zeta^{\delta}$
be a minimizer of problem \eqref{min_prob_noise}. Then, the following convergence result holds:
\[
\hat \zeta^{\delta} \to \hat \zeta \quad \text{as } \delta \to 0,
\]
and  \( \hat \zeta \)   is a   minimizer of problem \eqref{min_prob}.

\end{theorem}
\begin{proof}
The existence of minimizers follows from the fact that $J_\delta$ is continuous and the set $\mathcal{C}^\delta$
is compact. 

Now, we prove the convergence result. Let  $\hat \zeta^\delta\in \mathcal{C}^{\delta}$   be a minimizer of $J_\delta$.
By compactness and Hausdorff convergence,  we can extract a  subsequence  $\hat \zeta^{\delta_k}$  such that    $\hat \zeta^{\delta_k} \to  \hat\zeta$, where    $ \hat \zeta\in \mathcal  C$.
For any $\tilde \zeta\in \mathcal{C}$, there exists  $\tilde \zeta^{\delta}\in \mathcal{C^\delta}$ such  that 
$\tilde \zeta^{\delta}\rightarrow\tilde \zeta$.  Since   $\hat\zeta^\delta$  is a minimizer of  $J_\delta$ over  $\mathcal{C}^{\delta}$,  we have 
\[
  J_\delta(\hat \zeta^{\delta})  \leq J_\delta(\tilde \zeta^{\delta}).
\]
Passing to the limit in the obove equation, we obtain
\[
J(\hat\zeta)\leq J(\tilde\zeta).
\]
This implies that   $\hat\zeta$  is a minimizer of $J$.
\end{proof}
In the following numerical experiments, we use the same data as in Section~\ref{sec_mono_test}. The optimal solution is computed using the CVX toolbox in MATLAB, which is designed for modeling and solving convex optimization problems efficiently.
\subsection{Numerical results using the monotonicity constraints}
Figure~\ref{fig_suppSTSVD} illustrates the reconstruction of the common support of the parameter variations using monotonicity constraints. The first row shows the reconstruction obtained from noise-free data, while the second row displays the results with a noise level of $10\%$. As observed, the reconstruction is more accurate when the support consists of a single connected subdomain, compared to the case where it is composed of two disjoint subdomains.
 \begin{figure}[H]
\begin{center}
\subfloat{\includegraphics[scale=0.4]{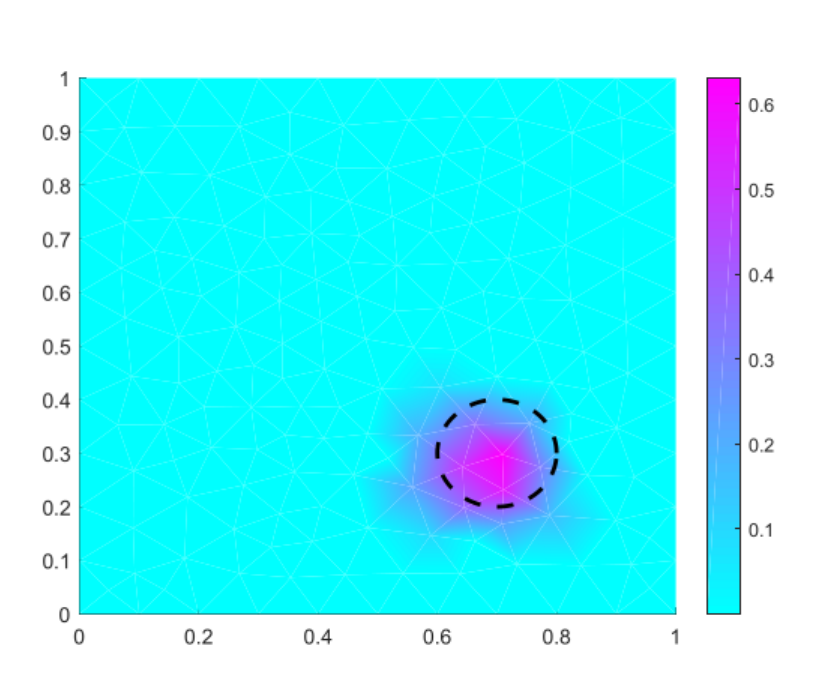}}
\subfloat{\includegraphics[scale=0.4]{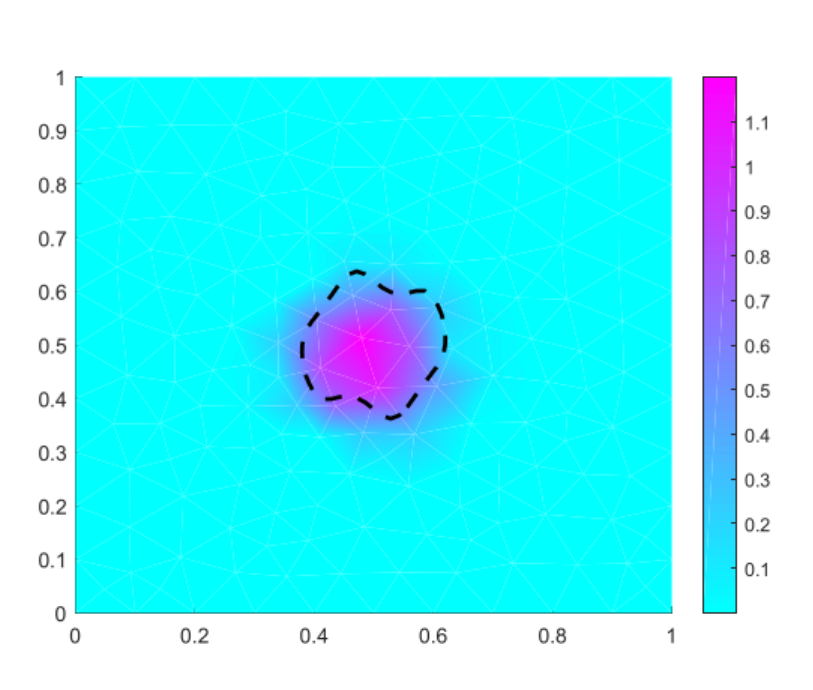}}
\subfloat{\includegraphics[scale=0.4]{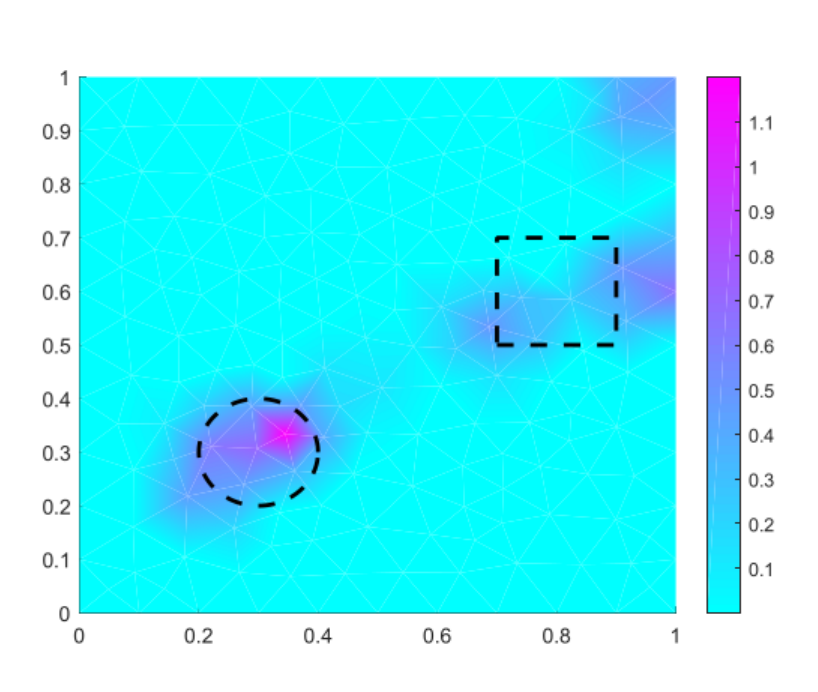}}\\
\subfloat{\includegraphics[scale=0.4]{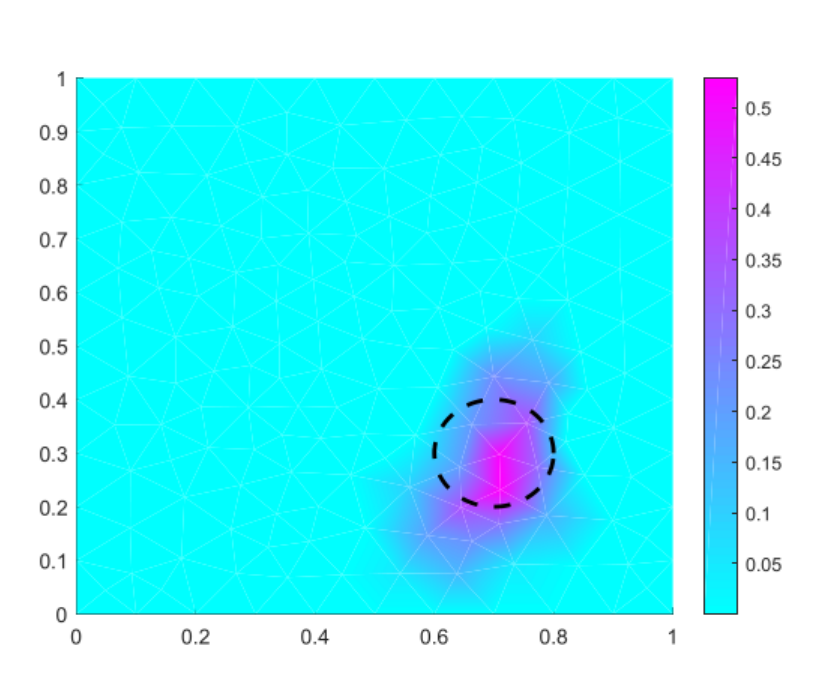}}
\subfloat{\includegraphics[scale=0.4]{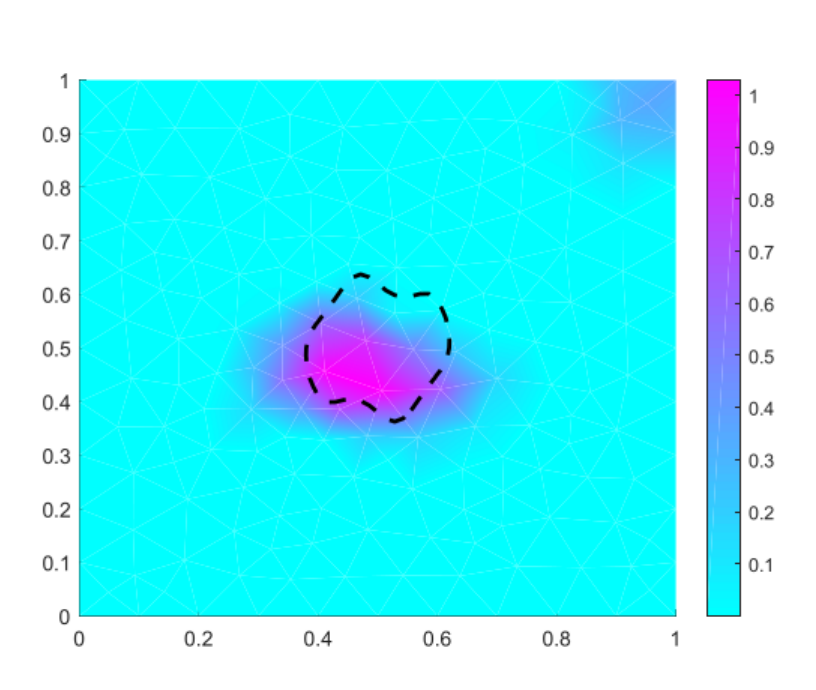}}
\subfloat{\includegraphics[scale=0.4]{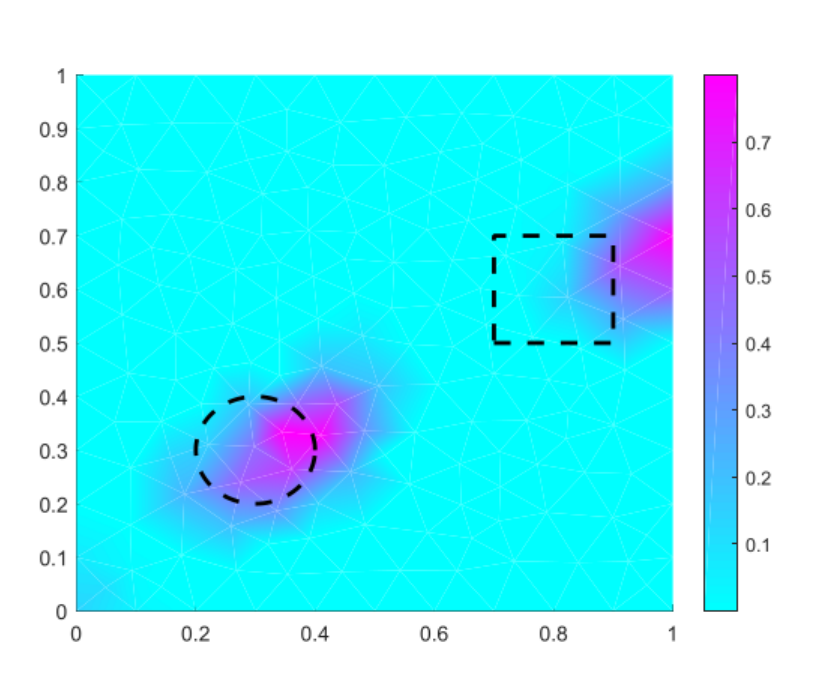}}
\caption{Reconstruction of the shape $D$ using the monotonicity constraints. The first row presents the reconstruction from noise-free data, while the second row shows the reconstruction from noisy data with a noise level of $\delta=0.1$ .} \label{fig_suppSTSVD}
\end{center}
\end{figure}

\section{Reconstruction with combined monotonicity and TSVD Regularization}
To reduce the numerical instability and enhance the numerical  reconstruction,  we use a reconstruction techniques 
based on a combination of the  monotonicity constraints method and the truncated singular value decomposition method(TSVD).

The singular value decomposition(SVD) of a matrix $A\in \mathbb{R}^{m\times n}$  is given by 
$A= \bold U\Sigma \bold V^T$,  with  $\bold U\in \mathbb{R}^{m\times m}$ and $\bold V\in \mathbb{R}^{m\times m}$ are orthogonal  matrices and $\Sigma= {\rm diag}(\sigma_1,\ldots \sigma_r)$ is a diagonal matrix containing the singular values $\sigma_1\geq \ldots\geq\sigma_r>0$ of $A$.

The \emph{Frobenius norm} of \( A \) can be expressed in terms of its singular values as:
\[
\|A\|_F^2 = \sum_{i=1}^{r} \sigma_i^2.
\]

To compute a low-rank approximation of \( A \), we seek the smallest integer \( k \) such that the partial sum of the  singular values captures a specified proportion of the total energy. This leads to the following criterion:
\[
\frac{\sum_{i=1}^{k} \sigma_i}{\sum_{i=1}^{r} \sigma_i} \geq \tau,
\]
where \( \tau \in (0,1) \) is a prescribed threshold.

The corresponding truncated SVD approximation of \( A \) is:
\[
A_l = U_k \Sigma_l V_l^T,
\]
where \( \Sigma_l = \mathrm{diag}(\sigma_1, \dots, \sigma_l) \), and \( U_l, V_l \) contain the first \( l \) columns of \( U \) and \( V \), respectively.

This approach is especially effective in eliminating small singular values associated with noise or ill-conditioning, thereby serving as a powerful tool for regularizing inverse problems.

In this subsection, rather than solving the minimization problem \eqref{min_prob}, we consider the following constrained TSVD  minimization problem
\begin{equation}\label{min_prob_trunc}
 \min_{\zeta \in \mathcal{C}}\Vert \tilde R(\zeta,\tau_1\zeta,\tau_2\zeta) \Vert^2_F,
\end{equation} 
where 
\[
\tilde R(\zeta) := \sum_{k=1}^L \alpha_k\tilde T^{\lambda}_k+\beta_k \tilde T^{\mu}_k+\gamma_k\tilde T^{\rho}_k-U,   
\]
with $[\tilde T^{\lambda}_k, \tilde T^{\mu}_k,   \tilde T^{\rho}_k]=  \tilde T_k$  where  
  $\tilde S_k$  is the truncation matrix of the matrix   $T_k=[T^{\lambda}_k,  T^{\mu}_k,    T^{\rho}_k]$.

For  noisy data, we solve  the following  constrained  minimization problem 
\begin{equation}\label{min_prob_trunc_noise}
 \min_{\zeta \in \mathcal{C}^{\delta}}\Vert \tilde R(\zeta,\tau_1\zeta,\tau_2\zeta) \Vert^2_F.
\end{equation} 
\subsection{Numerical results using TSVD  regularization}
Figure~\ref{TSVD_matrix} illustrates the exponential decay of the sensitivity matrix \( T_k \), the corresponding energy distribution, and the truncation point associated with the threshold \( \tau = 0.99 \).

Figure~\ref{combined_single_shape} presents the reconstruction of the common support of the parameter variations using the combined method and the TSVD regularization approach.

Figure~\ref{combined_multiple_shape} shows the reconstruction of disjoint supports of the parameter variations using the monotonicity method and the combined monotonicity--TSVD regularization method. The square indicates the support of \( \delta \), the circle denotes the support of \( \delta \rho \), and the ellipse represents the support of \( \delta \lambda \).

Compared to the monotonicity method alone, the combined approach yields a more accurate reconstruction.
 \begin{figure}[H]
\begin{center}
\includegraphics[scale=0.6]{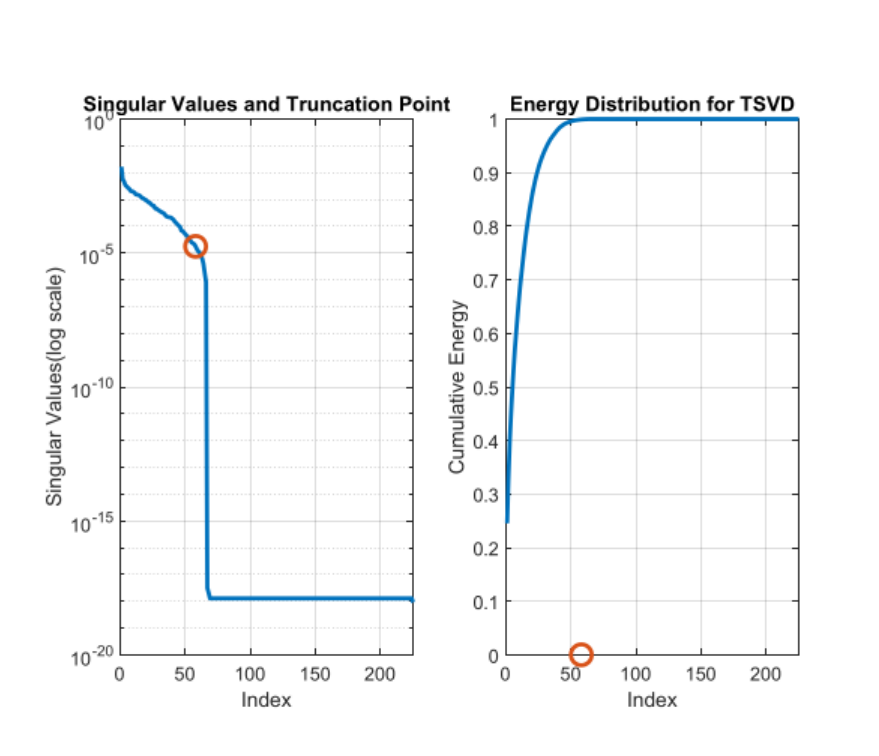}
\end{center}
\caption{Reconstruction of the truncation point (highlighted in red): on the left, the exponential decay of the singular values; on the right, the energy distribution for the TSVD.}\label{TSVD_matrix}
\end{figure}

  \begin{figure}[H]
\begin{center}
\subfloat{\includegraphics[scale=0.4]{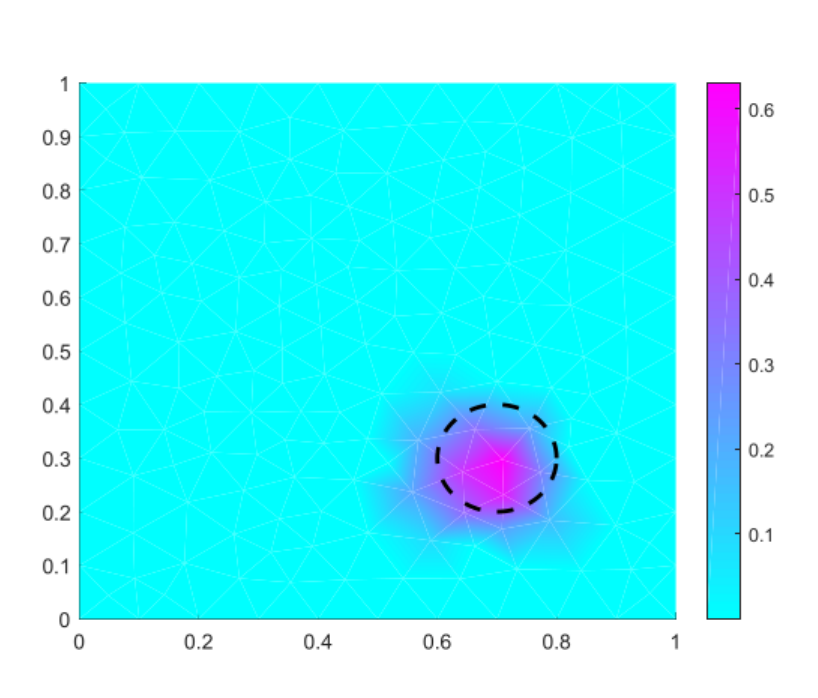}}
\subfloat{\includegraphics[scale=0.4]{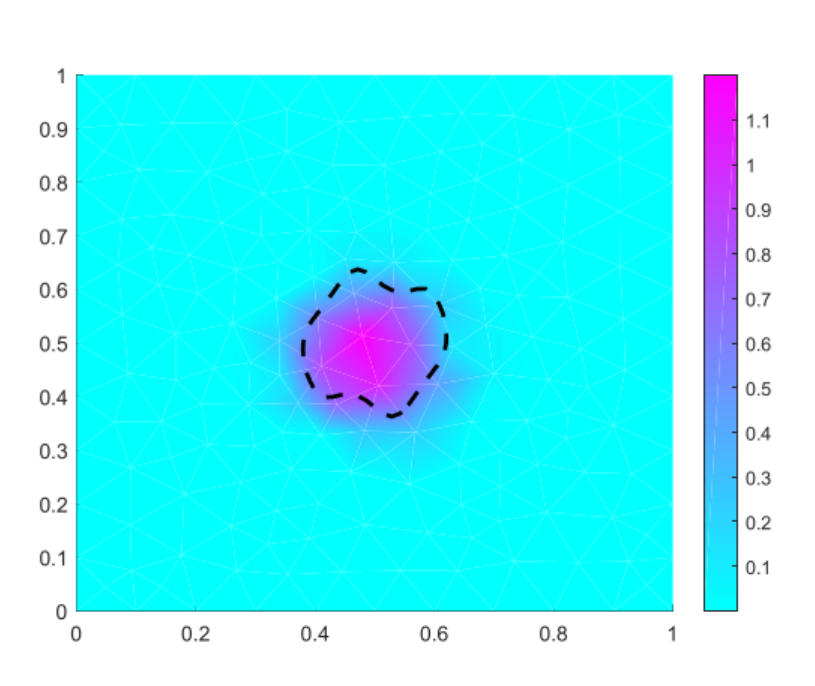}}
\subfloat{\includegraphics[scale=0.4]{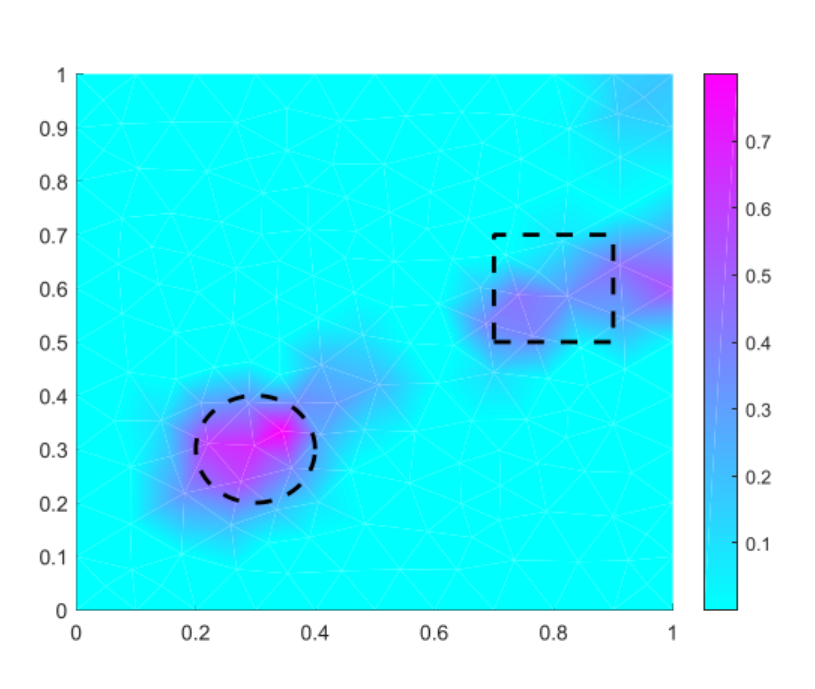}}\\
\subfloat{\includegraphics[scale=0.4]{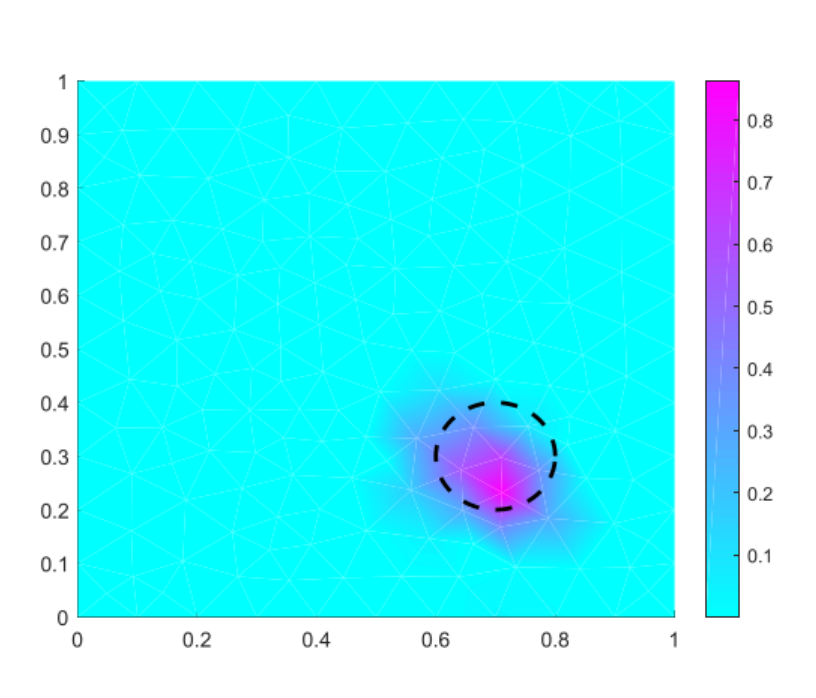}}
\subfloat{\includegraphics[scale=0.4]{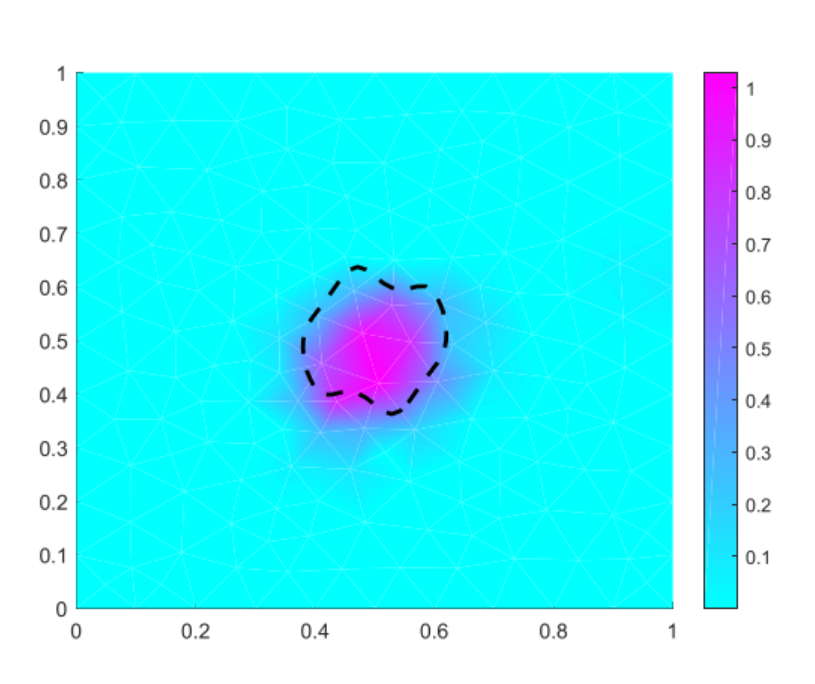}}
\subfloat{\includegraphics[scale=0.4]{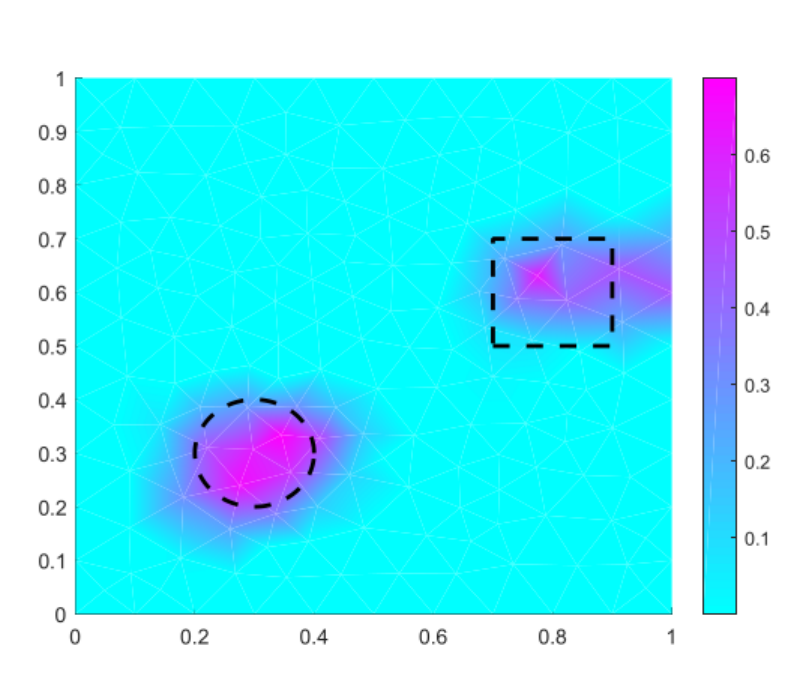}}
\caption{Reconstruction of the shape $D$ using combined monotonicity and  TSVD regularization. The first row presents the reconstruction from noise-free data, while the second row shows the reconstruction from noisy data with a noise level of $\delta=0.1$ .}\label{combined_single_shape}
\end{center}
\end{figure}

\begin{figure}[H]
\begin{center}
\subfloat{\includegraphics[scale=0.4]{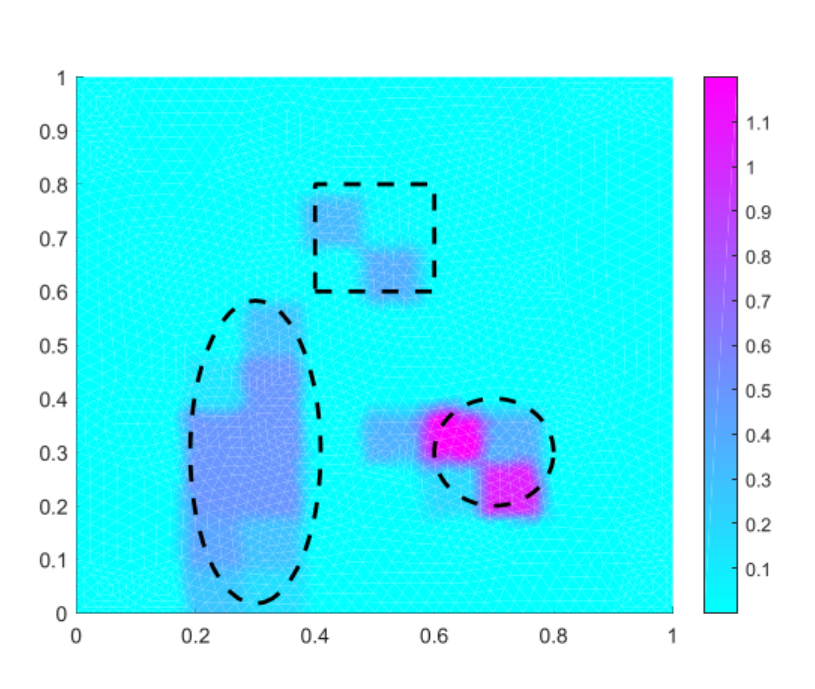}}
\subfloat{\includegraphics[scale=0.4]{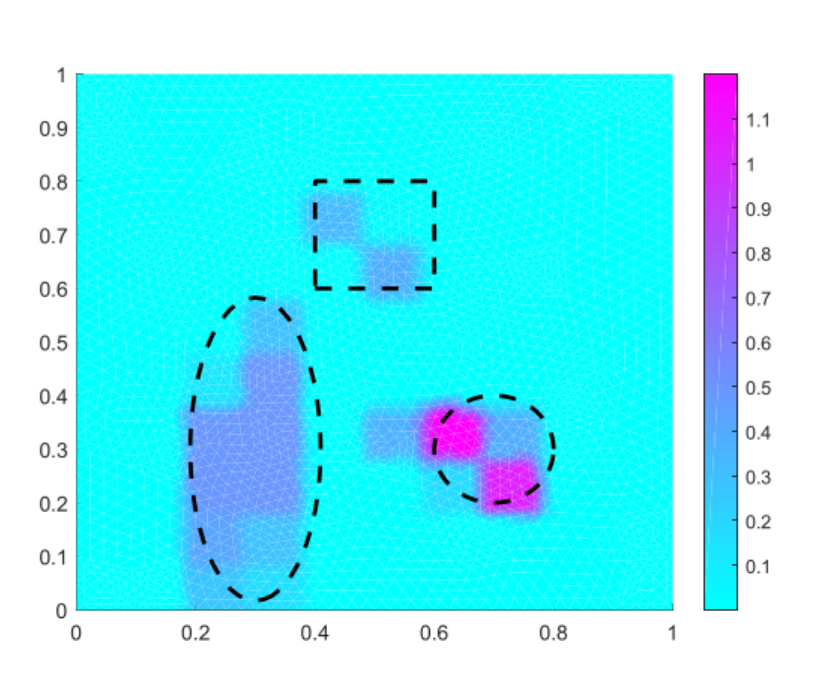}}\\
\subfloat{\includegraphics[scale=0.4]{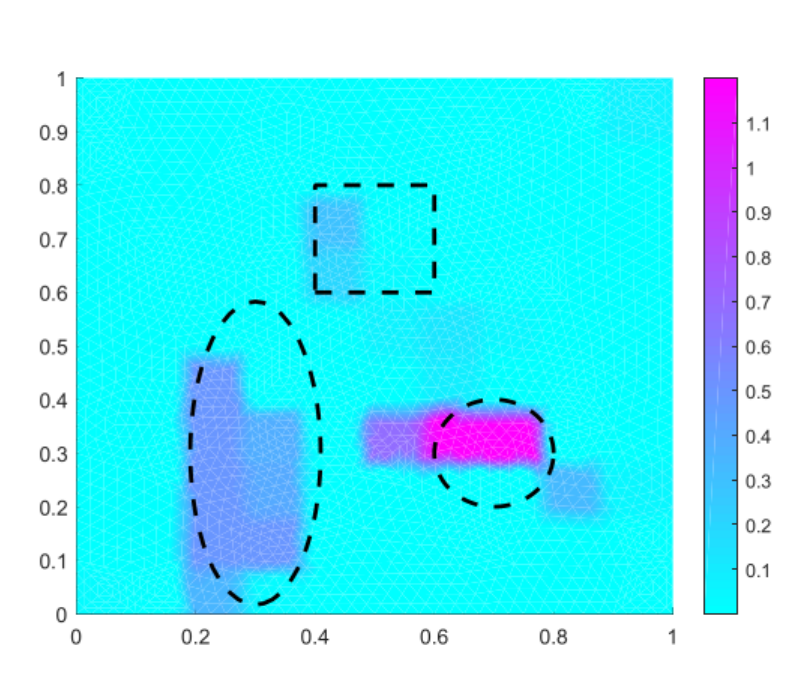}}
\subfloat{\includegraphics[scale=0.4]{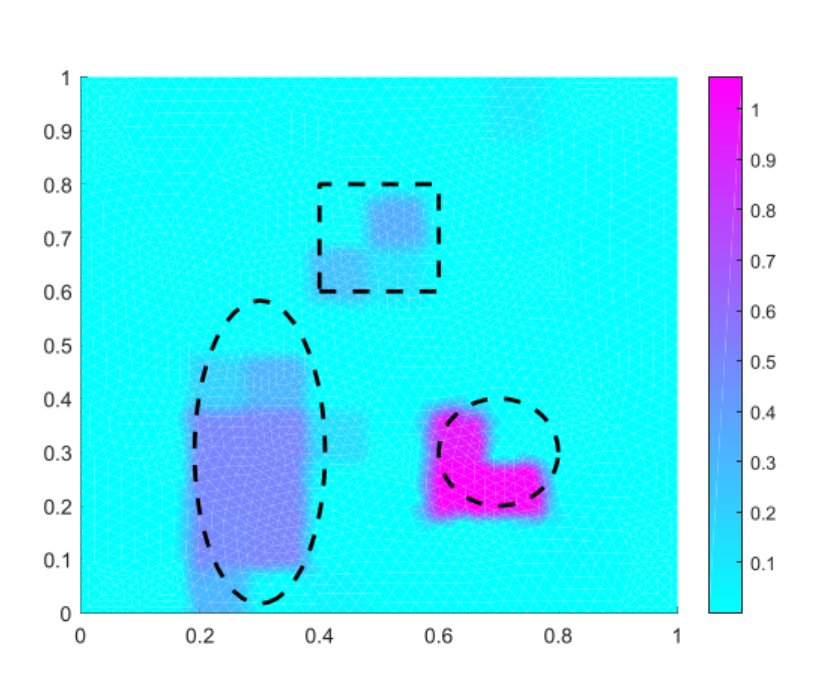}}
\caption{The first column displays, from top to bottom, the reconstructions of \( \mathrm{supp}\, \delta\lambda \), \( \mathrm{supp}\, \delta\mu \), and \( \mathrm{supp}\, \delta\rho \) using the monotonicity constraints, obtained from both noise-free data and noisy data with a noise level of \( \delta = 0.1 \). The second column displays, from top to bottom, the reconstructions of \( \mathrm{supp}\, \delta\lambda \), \( \mathrm{supp}\, \delta\mu \), and \( \mathrm{supp}\, \delta\rho \) using the combined monotonicity and TSVD method, also obtained from noise-free and noisy data with a noise level of \( \delta = 0.1 \).}\label{combined_multiple_shape}
\end{center}
\end{figure}
\section{Conclusion}
In this paper, we developed a numerical method that combines the monotonicity approach with the truncated singular value decomposition (TSVD) to address a geometric inverse problem.

We began by applying the linearized monotonicity test in the case where the parameter variations have a common  support. While effective, this method is sensitive to small amounts of noise. To improve the resolution, we combined the regularized monotonicity method with the TSVD technique. This hybrid approach proved particularly effective when the parameter variations have disjoint supports.

In future work, we aim to integrate the level set method with this approach in order to further refine the resolution and develop a globally convergent reconstruction method.
\section{Appendix}
\subsection{An abstract differentiability result}\label{app1}
In this section we give an abstract result for differentiating Lagrangian functionals with respect to a parameter. This result is used to prove Proposition \ref{prop_diff}. We first introduce some notations.  Consider the functional
\begin{equation}
G: [0,\varepsilon]\times X\times Y\rightarrow \mathbb{R}
\end{equation}
for some $\varepsilon>0$ and the Banach spaces $X$ and $Y$. For each
$t\in [0,\varepsilon]$,  define
\begin{equation}\label{s2}
g(t)=\adjustlimits\inf_{x\in X}\sup_{y\in Y}G(t,x,y),\qquad  h(t)=\adjustlimits \sup_{y\in Y}\inf_{x\in X}G(t,x,y),
\end{equation}
and the associated sets
\begin{align}
 X(t) &=\left\{ x^t\in X: \sup_{y\in Y}G(t,x^t,y)=g(t)\right\},\\
 Y(t) &=\left\{ y^t\in Y: \inf_{x\in X}G(t,x,y^t)=h(t)\right\}.
\end{align}
Note that inequality $h(t)\leq g(t)$ holds. If $h(t) = g(t)$ the set of saddle points is given by
\begin{equation}
S(t) := X(t)\times Y(t).
\end{equation}
We state now a simplified version of a result from \cite{MR948649} derived from \cite{CS} which gives realistic conditions that allows to differentiate $g(t)$ at $t=0$.  The main difficulty is to obtain conditions which allow to exchange the derivative with respect to $t$ and the inf-sup in \eqref{s2}.  
\begin{theorem}[Correa and Seeger\cite{CS,DZ2}]
\label{CS}
Let $X,Y,G$ and $\varepsilon$ be given as above. Assume that the following conditions hold :
\begin{itemize}
\item[(H1)] $S(t)\neq\emptyset$ for $0\leq t\leq \varepsilon$.
\item[(H2)] The partial derivative $\partial_t G(t,x,y)$ exists for all $(t,x,y)\in [0,\varepsilon]\times X\times Y $.
\item[(H3)] For any sequence $\{t_n\}_{n\in\mathbb{N}}$, with  $t_n\to 0$, there exist  a subsequence $\{t_{n_k}\}_{k\in\mathbb{N}}$ and $x^0\in X(0)$, $x_{n_k}\in X(t_{n_k})$ such that for all $y\in Y(0)$,
\begin{equation*}
\lim_{t\searrow 0, k\to\infty} \partial_t G(t,x_{n_k},y) = \partial_t G(0,x^0,y),
\end{equation*}
\item[(H4)] For any sequence $\{t_n\}_{n\in\mathbb{N}}$, with  $t_n\to 0$, there exist  a subsequence $\{t_{n_k}\}_{k\in\mathbb{N}}$ and $y^0\in Y(0)$, $y_{n_k}\in Y(t_{n_k})$ such that for all $x\in X(0)$,
\begin{equation*}
\lim_{t\searrow 0, k\to\infty} \partial_t G(t,x,y_{n_k}) = \partial_t G(0,x,y^0),
\end{equation*}
\end{itemize}
Then there exists $(x^0,y^0)\in X(0)\times Y(0)$  such that
\begin{align*}
\frac{dg}{dt}(0) &=  \partial_t G(0,x^0,y^0).
\end{align*}
\end{theorem}

\bibliographystyle{plain}
\bibliography{biblio}

\begin{thebibliography}{10}

\bibitem{belhachmi2018topology}
Zakaria Belhachmi, A~Ben~Abda, B~Meftahi, and Houcine Meftahi.
\newblock Topology optimization method with respect to the insertion of small
  coated inclusion.
\newblock {\em Asymptotic Analysis}, 106(2):99--119, 2018.

\bibitem{belhachmi2023level}
Zakaria Belhachmi, Rabeb Dhif, and Houcine Meftahi.
\newblock Level set-based shape optimization approach for the inverse optical
  tomography problem.
\newblock {\em ZAMM-Journal of Applied Mathematics and Mechanics/Zeitschrift
  f{\"u}r Angewandte Mathematik und Mechanik}, 103(3):e202200156, 2023.

\bibitem{belhachmi2013shape}
Zakaria Belhachmi and Houcine Meftahi.
\newblock Shape sensitivity analysis for an interface problem via minimax
  differentiability.
\newblock {\em Applied Mathematics and Computation}, 219(12):6828--6842, 2013.

\bibitem{calvano2012fast}
Flavio Calvano, Guglielmo Rubinacci, and Antonello Tamburrino.
\newblock Fast methods for shape reconstruction in electrical resistance
  tomography.
\newblock {\em NDT and $\&$ E International}, 46:32--40, 2012.

\bibitem{chaabane2013topological}
S~Chaabane, Mohamed Masmoudi, and Houcine Meftahi.
\newblock Topological and shape gradient strategy for solving geometrical
  inverse problems.
\newblock {\em Journal of Mathematical Analysis and applications},
  400(2):724--742, 2013.

\bibitem{CS}
Rafael Correa and Alberto Seeger.
\newblock Directional derivative of a minimax function.
\newblock {\em Nonlinear Anal.}, 9(1):13--22, 1985.

\bibitem{MR948649}
M.~C. Delfour and J.-P. Zol{\'e}sio.
\newblock Shape sensitivity analysis via min max differentiability.
\newblock {\em SIAM J. Control Optim.}, 26(4):834--862, 1988.

\bibitem{DZ2}
M.~C. Delfour and J.-P. Zol{\'e}sio.
\newblock {\em Shapes and geometries}, volume~22 of {\em Advances in Design and
  Control}.
\newblock Society for Industrial and Applied Mathematics (SIAM), Philadelphia,
  PA, second edition, 2011.
\newblock Metrics, analysis, differential calculus, and optimization.

\bibitem{eberle2021shape}
Sarah Eberle and Bastian Harrach.
\newblock Shape reconstruction in linear elasticity: standard and linearized
  monotonicity method.
\newblock {\em Inverse Problems}, 37(4):045006, 2021.

\bibitem{eberle2022monotonicity}
Sarah Eberle and Bastian Harrach.
\newblock Monotonicity-based regularization for shape reconstruction in linear
  elasticity.
\newblock {\em Computational mechanics}, 69(5):1069--1086, 2022.

\bibitem{eberle2023resolution}
Sarah Eberle-Blick and Bastian Harrach.
\newblock Resolution guarantees for the reconstruction of inclusions in linear
  elasticity based on monotonicity methods.
\newblock {\em Inverse Problems}, 39(7):075006, 2023.

\bibitem{ekeland1974analyse}
Ivar Ekeland and Roger Temam.
\newblock Analyse convexe et problemes variationnels.
\newblock {\em (No Title)}, 1974.

\bibitem{esposito2024piecewise}
Antonio~Corbo Esposito, Luisa Faella, Vincenzo Mottola, Gianpaolo Piscitelli,
  Ravi Prakash, and Antonello Tamburrino.
\newblock Piecewise nonlinear materials and monotonicity principle.
\newblock {\em Inverse Problems}, 40(8):085001, 2024.

\bibitem{garde2022simplified}
Henrik Garde.
\newblock Simplified reconstruction of layered materials in eit.
\newblock {\em Applied Mathematics Letters}, 126:107815, 2022.

\bibitem{garde2022reconstruction}
Henrik Garde and Nuutti Hyv{\"o}nen.
\newblock Reconstruction of singular and degenerate inclusions in
  calder{\'o}n's problem.
\newblock {\em Inverse Problems and Imaging}, 16(5):1219--1227, 2022.

\bibitem{griesmaier2022inverse}
Roland Griesmaier, Marvin Kn{\"o}ller, and Rainer Mandel.
\newblock Inverse medium scattering for a nonlinear helmholtz equation.
\newblock {\em Journal of Mathematical Analysis and Applications},
  515(1):126356, 2022.

\bibitem{harrach2013monotonicity}
Bastian Harrach and Marcel Ullrich.
\newblock Monotonicity-based shape reconstruction in electrical impedance
  tomography.
\newblock {\em SIAM Journal on Mathematical Analysis}, 45(6):3382--3403, 2013.

\bibitem{klibanov2025convexification}
Michael~V Klibanov, Jingzhi Li, and Zhipeng Yang.
\newblock Convexification with the viscocity term for electrical impedance
  tomography.
\newblock {\em arXiv preprint arXiv:2503.07916}, 2025.

\bibitem{laurain2016shape}
Antoine Laurain and Houcine Meftahi.
\newblock Shape and parameter reconstruction for the robin transmission inverse
  problem.
\newblock {\em Journal of Inverse and Ill-posed Problems}, 24(6):643--662,
  2016.

\bibitem{lin2006quantitative}
C.-L. Lin, G.~Nakamura, Gunther Uhlmann, and J.-N. Wang.
\newblock Quantitative strong unique continuation for the lamé system with
  lipschitz coefficients.
\newblock {\em Methods and Applications of Analysis}, 13(2):183--198, 2006.

\bibitem{lin2022monotonicity}
Yi-Hsuan Lin.
\newblock Monotonicity-based inversion of fractional semilinear elliptic
  equations with power type nonlinearities.
\newblock {\em Calculus of Variations and Partial Differential Equations},
  61(5):188, 2022.

\bibitem{meftahi2015sensitivity}
H~Meftahi and J-P Zol{\'e}sio.
\newblock Sensitivity analysis for some inverse problems in linear elasticity
  via minimax differentiability.
\newblock {\em Applied Mathematical Modelling}, 39(5-6):1554--1576, 2015.

\bibitem{meftahi2009etudes}
Houcine Meftahi.
\newblock {\em {\'E}tudes th{\'e}oriques et num{\'e}riques de quelques
  probl{\`e}mes inverses}.
\newblock PhD thesis, Lille 1, 2009.

\bibitem{meftahi2025stability}
Houcine Meftahi and Chayma Nssibi.
\newblock Stability analysis of an inverse coefficients problem in a system of
  partial differential equations.
\newblock {\em arXiv preprint arXiv:2505.05116}, 2025.

\bibitem{mottola2024imaging}
Vincenzo Mottola, Antonio~Corbo Esposito, Gianpaolo Piscitelli, and Antonello
  Tamburrino.
\newblock Imaging of nonlinear materials via the monotonicity principle.
\newblock {\em Inverse Problems}, 40(3):035007, 2024.

\bibitem{Tamburrino}
A.~Tamburrino.
\newblock Monotonicity based imaging methods for elliptic and parabolic inverse
  problems.
\newblock {\em Journal of Numerical Mathematics}, 14(6):633--642, 2006.

\bibitem{tamburrino2002new}
Antonello Tamburrino and Guglielmo Rubinacci.
\newblock A new non-iterative inversion method for electrical resistance
  tomography.
\newblock {\em Inverse Problems}, 18(6):1809, 2002.

\bibitem{yu2011strong}
Hang Yu.
\newblock Strong unique continuation for the lam{\'e} system with lipschitz
  coefficients in three dimensions.
\newblock {\em ESAIM: Control, Optimisation and Calculus of Variations},
  17(3):761--770, 2011.

\end{thebibliography}

\end{document}